\documentclass[12pt]{article}
\usepackage[english]{babel}
\usepackage[cp1250]{inputenc}
\usepackage{amsfonts, amsmath, amssymb}
\usepackage{amsthm}

\newtheorem{thm}{Theorem}
\newtheorem{prop}[thm]{Proposition}
\newtheorem{lemma}[thm]{Lemma}
\newtheorem{cor}[thm]{Corollary}

\theoremstyle{definition}
\newtheorem*{defn}{Definition}
\newtheorem*{rmk}{Remark}

\begin{document}

\begin{center}
{\Large \bf On varieties of commuting nilpotent matrices}
\end{center}
\begin{center}
{\large Nham V. Ngo}\\
Department of Mathematics and Statistics, Lancaster University, Lancaster, LA1 4YW, United Kingdom\\e-mail address: n.ngo@lancaster.ac.uk
\end{center}
\begin{center}and\end{center}
\begin{center}
{\large Klemen \v{S}ivic}\\
Department of Mathematics, ETH Z\"urich, R\"amistrasse 101, Z\"urich, Switzerland \\ e-mail address: klemen.sivic@math.ethz.ch
\end{center}

\begin{abstract}
Let $N(d,n)$ be the variety of all $d$-tuples of commuting nilpotent $n\times n$ matrices. It is well-known that $N(d,n)$ is irreducible if $d=2$, if $n\le 3$ or if $d=3$ and $n=4$. On the other hand $N(3,n)$ is known to be reducible for $n\ge 13$. We study in this paper the reducibility of $N(d,n)$ for various values of $d$ and $n$. In particular, we prove that $N(d,n)$ is reducible for all $d,n\ge 4$. In the case $d=3$, we show that it is irreducible for $n\le 6$. 
{\bf Math. subj. class.:} 15A27
\end{abstract}

\section{Introduction}
Let $C(d,n)$ denote the set of all $d$-tuples of commuting $n\times n$ matrices over an algebraically closed field $\mathbb{F}$ and let $N(d,n)$ be its subset consisting of all $d$-tuples of {\em nilpotent} commuting matrices. Both sets are defined by polynomial equations in the entries of matrices, therefore they can be viewed as affine varieties in $\mathbb{F}^{dn^2}$. Irreducibility of $C(d,n)$ has been studied for a long time. Motzkin and Taussky \cite{MT} proved that the variety $C(2,n)$ is irreducible for each positive integer $n$. On the other hand, $C(d,n)$ is known to be reducible if $d$ and $n$ are both at least 4. This result is usually referred to Gerstenhaber \cite{Ge}, but he proved reducibility of $C(d,n)$ only if $n\ge 4$ and $d\ge n+1$ (and his proof can be applied also to $d=n\ge 4$), while reducibility of $C(d,n)$ for $d\ge 4$ and $n\ge 4$ was proved by Guralnick \cite{Gu}, using Gerstenhaber's idea. It is also known that $C(d,n)$ is irreducible for each $d$ if $n\le 3$ (see \cite{KN} and \cite{Gu}). The situation in the case of triples is much more complicated, and the problem of irreducibility of $C(3,n)$ is not solved completely. With a dimension argument Guralnick \cite{Gu} proved that $C(3,n)$ is reducible for $n\ge 32$, and using his idea Holbrook and Omladi\v{c} showed reducibility of $C(3,n)$ for $n\ge 30$. In fact, their proof in \cite{HO} shows also reducibility of $C(3,29)$, since in the case of irreducibility of $C(3,n)$ all its subvarieties must have strictly smaller dimension than $n^2+2n$, and the inequality at the beginning of page 136 of \cite{HO} does not need to be strict. On the other hand, $C(3,n)$ is known to be irreducible for $n\le 10$ (see \cite{GS, HO, O, Han, S1, S3}). The last results were stated if $\mathrm{char}\, \mathbb{F}=0$, but in many proofs this assumption can be omitted, and probably slight modifications of the other proofs would also give the same results in any characteristic.

In this paper we study reducibility of $N(d,n)$ and the characteristic of $\mathbb{F}$ will be arbitrary. The study of $N(d,n)$ started much later than that of $C(d,n)$. Baranovsky \cite{Bar} and Basili \cite{Bas} independently showed that $N(2,n)$ is irreducible for each $n$, and this result was significantly generalized by Premet \cite{P} to varieties of pairs of commuting nilpotent elements of arbitrary reductive Lie algebras. Explicitly, he proved that these varieties are equidimensional and their components are parametrized by distinguished nilpotent conjugacy classes in the Lie algebra. Moreover, his proof for irreducibility of $N(2,n)$ does not depend on the characteristic of $\mathbb{F}$, while in \cite{Bar} and \cite{Bas} there are some restrictions on $\mathrm{char}\, \mathbb{F}$. On the other hand, for $d\ge 3$ much less is known. If $n\le 3$, then $N(d,n)$ is irreducible for each $d$ by Proposition 5.1.1 and Theorem 7.1.2 of \cite{N}. In the introduction of Chapter II of Young's thesis \cite{Y} it was claimed that $N(d,n)$ is reducible for $d\ge 4$ and $n\ge 4$. However, this result was not proved and it is also not clear how to obtain it. To show reducibility of $C(d,n)$ for $n\ge 4$ and $d\ge n$ Gerstehnaber \cite{Ge} constructed a unital algebra (i.e. an algebra with an identity) of dimension more than $n$ which is generated by $d$ commuting (square zero) $n\times n$ matrices, and his proof can clearly be applied to the nilpotent case to show reducibility of $N(d,n)$ for $d\ge n\ge 4$. However, Guralnick's \cite{Gu} proof of reducibility of $C(d,n)$ for $n\ge 4$ and $d\ge 4$, using block diagonal matrices, cannot be applied to the nilpotent case, since all blocks have to be nilpotent. Using Gerstenhaber's idea we will prove that $N(d,n)$ is indeed reducible for all $d\ge 4$ and $n\ge 4$. In particular, for each $n\ge 4$ we will find a unital $(n+1)$-dimensional commutative subalgebra of $M_n(\mathbb{F})$ generated by 4 nilpotent matrices. The question of irreducibility of $N(d,n)$ then remains open only in the case of triples, i.e. for $d=3$. There are two known results about irreducibility of $N(3,n)$. In \cite{Y}, Theorem 2.2.1 it was proved  that $N(3,4)$ is irreducible. On the other hand, Clark, O'Meara and Vinsonhaler proved in their book on Weyr form \cite{CO'MV} that the variety $N(3,n)$ is reducible for $n\ge 13$ (Theorem 7.10.5), using methods from \cite{Gu} and \cite{HO}. Since this result was proved for much smaller dimension than analogous result for $C(3,n)$, it seems to be easier to determine when $N(3,n)$ is irreducible than when $C(3,n)$ is. In this paper we partially answer this question. We prove irreducibility of $N(3,n)$ for $n\le 6$. Our result provides a rigorous argument for irreducibility of $N(3,4)$ stated in Young's thesis \cite{Y}. We also prove that the variety of pairs of commuting nilpotent matrices in the centralizer of a 2-regular nilpotent matrix (i.e. matrix whose Jordan canonical form has at most 2 Jordan blocks) is irreducible, which gives an analog to Corollary 10 of \cite{NS}. However, we will show that the same result does not hold for 3-regular case, i.e. a result analogous to Theorem 12 of \cite{S2} does not hold in the nilpotent case.

The paper is organized as follows. In Section 2 we recall some preliminary results on nilpotent commuting matrices from the literature. We show that the variety $N(d,n)$ is irreducible if and only if it is equal to the Zariski closure of the set of all $d$-tuples of commuting nilpotent 1-regular $n\times n$ matrices. As a corollary we show two interesting relations between (ir)reducibility of varieties $N(d,n)$ and $C(d,n)$. In Section 3 we show that for each $n\ge 4$ there exists a unital $(n+1)$-dimensional commutative subalgebra of $M_n(\mathbb{F})$ which is generated by 4 nilpotent matrices. This will imply reducibility of varieties $N(d,n)$ for all $d\ge 4$ and $n\ge 4$. In Section 4 we consider varieties of pairs of commuting nilpotent matrices in the centralizers of given nilpotent matrices. We show that such variety is irreducible if the given matrix is 2-regular, but it can be reducible if the given matrix is 3-regular. In Section 5 we prove irreducibility of varieties $N(3,5)$ and $N(3,6)$, using simultaneous commutative perturbations of triples of commuting nilpotent matrices by triples of commuting nilpotent 1-regular matrices. We also show that in any dimension a triple of nilpotent commuting matrices can be perturbed by triples of 1-regular nilpotent commuting matrices if Jordan canonical forms of the matrices in the triple have only one nonzero Jordan block or all Jordan blocks of order at most 2. Note that analogous results in the case of $C(3,n)$ were proved in \cite{HO}.

\section{Preliminaries}

In the proofs of (ir)reducibility of $C(d,n)$ generic matrices play an important role. Recall that an $n\times n$ matrix is called {\em generic} if it has $n$ distinct eigenvalues. Generic matrices form an open subset of $M_n(\mathbb{F})$ defined by the inequality $\det p_X'(X)\ne 0$ where $p'_X$ is the derivative of the characteristic polynomial of the matrix $X$. In particular, each irreducible component of the subvariety of $M_n(\mathbb{F})$ consisting of all matrices that are not generic is $(n^2-1)$-dimensional. Moreover, the variety $C(d,n)$ is irreducible if and only if it is equal to the Zariski closure of the set of all $d$-tuples of generic commuting $n\times n$ matrices (see \cite{S3}, and note that these results do not depend on the characteristic of $\mathbb{F}$). To adapt this strategy for proving irreducibility of $N(d,n)$, we use the concept of 1-regular matrices instead. In this section we recall some results from the literature, especially from \cite{CO'MV}, on commuting nilpotent matrices, that will be needed in the sequel. We start with the following definition from \cite{NS}.

\begin{defn}
A matrix $A\in M_n(\mathbb{F})$ is called {\em $r$-regular} if each its eigenspace is at most $r$-dimensional.
\end{defn}

\begin{rmk}
Note that the eigenspace of a nilpotent matrix is its kernel and that a nilpotent matrix is $r$-regular if and only if its Jordan canonical form has at most $r$ Jordan blocks.
\end{rmk}

For the study of commutativity 1-regular matrices (which are also called {\em regular} or {\em nonderogatory}) are especially important, since each matrix that commutes with 1-regular matrix is a polynomial in that matrix. Moreover, each (nilpotent) $n\times n$ matrix can be perturbed by (nilpotent) 1-regular matrices, and therefore it belongs to the Zariski closure of the set of (nilpotent) 1-regular matrices. We denote by $N_n$ the variety of all nilpotent $n\times n$ matrices, by $R_n$ the set of all 1-regular nilpotent $n\times n$ matrices, by $R(d,n)$ the set of all $d$-tuples $(A_1,A_2,\ldots ,A_d)\in N(d,n)$ such that the linear span of $A_1,\ldots ,A_d$ contains a 1-regular matrix, and by $R_1(d,n)$ the set of all $d$-tuples $(A_1,\ldots ,A_d)\in N(d,n)$ with $A_1$ 1-regular. By Proposition 1 of \cite{NS} the set of all 1-regular $n\times n$ matrices is open in $M_n(\mathbb{F})$ in the Zariski topology, therefore $R_n$ is open in $N_n$ and $R_1(d,n)$ is open in $N(d,n)$. On the other hand, the set $R(d,n)$ is obtained by the action of the general linear group $GL_d(\mathbb{F})$ on $R_1(d,n)$, hence it can be written as a union of the sets isomorphic to $R_1(d,n)$. Therefore $R(d,n)$ is also open in $N(d,n)$. In particular, if $N(d,n)$ is irreducible, then $\overline{R(d,n)}=\overline{R_1(d,n)}=N(d,n)$, where $\overline{S}$ denotes the closure of the set $S$ in the Zariski topology. We first show that the converse of this implication also holds. If follows from the following proposition (which is essentially Lemma 7.10.3 of \cite{CO'MV}) that shows the irreducibility and the dimension of $\overline{R_1(d,n)}$.

\begin{prop}\label{1-regular}
The variety $\overline{R_1(d,n)}$ is irreducible and of dimension $(n+d-1)(n-1)$.
\end{prop}
\begin{proof}
For any nonnegative integer $k$ we denote by $\mathbb{F}_k[t]$ the space of all polynomials in $t$ over $\mathbb{F}$ of degree at most $k$. The polynomial map $\varphi \colon R_n\times (\mathbb{F}_{n-2}[t])^{d-1}\to R_1(d,n)$ defined by
$$\varphi (A,p_1,p_2,\ldots ,p_{d-1})=(A,Ap_1(A),Ap_2(A),\ldots ,Ap_{d-1}(A))$$
is bijective, since all matrices that commute with 1-regular nilpotent matrix $A$ are polynomials in $A$, and the constant terms of these polynomials have to be zero because of the nilpotency of the matrices in a $d$-tuple from $N(d,n)$. The variety $N_n=\overline{R_n}$ is irreducible and of dimension $n^2-n$ by Proposition 2.1 of \cite{Bas}, therefore $\overline{R_1(d,n)}$ is the closure of a polynomial image of the irreducible variety $N_n\times (\mathbb{F}_{n-2}[t])^{d-1}$, hence it is irreducible. Moreover, the theorem on fibres (Theorem 11.12 of \cite{Har}) implies that the varieties $N_n\times (\mathbb{F}_{n-2}[t])^{d-1}$ and $\overline{R_1(d,n)}$ have the same dimensions, i.e.
$$\dim \overline{R_1(d,n)}=\dim N_n+\dim ((\mathbb{F}_{n-2}[t])^{d-1})=(n+d-1)(n-1).$$
\end{proof}

\begin{rmk}
Based on this result and the arguments from \cite{Gu} and \cite{HO} the reducibility of $N(3,n)$ for $n\ge 13$ was proved in Theorem 7.10.5 of \cite{CO'MV}.
\end{rmk}

\begin{cor}\label{dense}
The variety $N(d,n)$ is irreducible if and only if $\overline{R(d,n)}=N(d,n)$.
\end{cor}
\begin{proof}
Since the variety $\overline{R_1(d,n)}$ is irreducible and the set $R(d,n)$ is obtained from $R_1(d,n)$ by the action of the group $GL_d(\mathbb{F})$, the variety $\overline{R(d,n)}$ is also irreducible. Moreover, since the set $R_1(d,n)$ is its open subset, the varieties $\overline{R(d,n)}$ and $\overline{R_1(d,n)}$ are equal. The corollary now follows immediately from the previous proposition.
\end{proof}

The following two corollaries show the relations between (ir)reducibility of $C(d,n)$ and of $N(d,n)$.

\begin{cor}\footnote{This is originally Theorem 3.2.2 in the unpublished paper \cite{GN}. The first author thanks Robert Guralnick for sharing the idea with him.}
Suppose that the variety $C(d,m)$ is irreducible for each $m<n$ and that $N(d,n)$ is irreducible. Then $C(d,n)$ is also irreducible.
\end{cor}
\begin{proof}
We have to show that $\overline{R^1(d,n)}=C(d,n)$, where $R^1(d,n)$ is the set of all $d$-tuples $(A_1,\ldots ,A_d)\in C(d,n)$ where at least one of the matrices $A_1,\ldots ,A_d$ is 1-regular (see \cite{S2}). Let $(A_1,\ldots ,A_d)\in C(d,n)$ be an arbitrary $d$-tuple. If one of the matrices $A_1,\ldots ,A_d$ has two distinct eigenvalues, then extending the proof of Lemma 5 of \cite{S3} (which clearly holds in any characteristic) to any number of matrices we get $(A_1,\ldots ,A_d)\in \overline{R^1(d,n)}$. In the sequel we can therefore assume that $A_i$ has only one eigenvalue $\lambda _i$ for each $i=1,\ldots ,d$. If we denote by $I$ the identity matrix, then $(A_1-\lambda _1I,\ldots ,A_d-\lambda _dI)\in N(d,n)=\overline{R_1(d,n)}\subseteq \overline{R^1(d,n)}$ by the assumption on irreducibility of the variety $N(d,n)$. However, the sets $R^1(d,n)$ and $\{(X_1+\lambda _1I,\ldots ,X_d+\lambda _dI);(X_1,\ldots ,X_d)\in R^1(d,n)\}$ are clearly the same, so their closures also are, and we get $(A_1,\ldots ,A_d)\in \overline{R^1(d,n)}$.
\end{proof}

\begin{cor}
If the variety $N(d,m)$ is irreducible for each $m\le n$, then the variety $C(d,n)$ is also irreducible.
\end{cor}

Since we know that $C(4,4)$ is reducible and $N(4,n)$ is irreducible for $n\le 3$, we obtain:

\begin{cor}\label{44}
The variety $N(4,4)$ is reducible.
\end{cor}

We conclude this section with two lemmas telling us which reductions we can do while proving that $N(d,n)=\overline{R(d,n)}$. The first lemma is a simple consequence of the fact that in an irreducible variety any nonempty open subset is dense.

\begin{lemma}\label{irr-open}
If $\mathcal{V}\subseteq N(d,n)$ is any irreducible variety and $\mathcal{U}\subseteq \mathcal{V}$ any nonempty open subset such that $\mathcal{U}\subseteq \overline{R(d,n)}$, then $\mathcal{V}\subseteq \overline{R(d,n)}$.
\end{lemma}

The lemma tells us that if we want to prove that any $d$-tuple from some irreducible subvariety of $N(d,n)$ belongs to $\overline{R(d,n)}$, then we can assume any nonempty open condition on the $d$-tuple.

\begin{lemma}\label{reduction}
Let $(A_1,\ldots ,A_d)\in N(d,n)$ be a $d$-tuple of commuting nilpotent matrices, $P\in GL_n(\mathbb{F})$ an arbitrary invertible matrix, $Q\in GL_n(\mathbb{F})$ an invertible matrix satisfying $Q^{-1}A_1^TQ=A_1$, and $p_2,\ldots ,p_d\in \mathbb{F}[t]$ arbitrary polynomials with zero constant terms. Then $(A_1,\ldots ,A_d)\in \overline{R(d,n)}$ if and only if any of the following holds.
\begin{enumerate}
\item[(a)]
$(P^{-1}A_1P,\ldots ,P^{-1}A_dP)\in \overline{R(d,n)}$.
\item[(b)]
$(B_1,\ldots ,B_d)\in \overline{R(d,n)}$ where $(B_1,\ldots ,B_d)\in N(d,n)$ is any $d$-tuple of nilpotent commuting matrices such that the linear spans of $A_1,\ldots ,A_d$ and of $B_1,\ldots ,B_d$ are the same.
\item[(c)]
$(A_1,A_2-p_2(A_1),\ldots ,A_d-p_d(A_1))\in \overline{R(d,n)}$.
\item[(d)]
$(A_1^T,\ldots ,A_d^T)\in \overline{R(d,n)}$.
\item[(e)]
$(A_1,Q^{-1}A_2^TQ,\ldots ,Q^{-1}A_d^TQ)\in \overline{R(d,n)}$.
\end{enumerate}
\end{lemma}
\begin{proof}
Let $\varphi \colon N(d,n)\to N(d,n)$ be a polynomial map that maps the set $R(d,n)$ to its closure. Since polynomial maps are continuous in the Zariski topology, it follows that $\varphi (\overline{R(d,n)})\subseteq \overline{\varphi (R(d,n))}\subseteq \overline{R(d,n)}$. Moreover, if $\varphi$ is invertible and its inverse is also a polynomial map that maps the set $R(d,n)$ to its closure, then $\varphi (\overline{R(d,n)})=\overline{R(d,n)}$. Now the equivalences in (a), (b), (d) and (e) follow immediately if the map $\varphi$ is respectively conjugation of the matrices in the $d$-tuple by $P$, the action of a fixed element of $GL_d(\mathbb{F})$ on the $d$-tuple that maps $(A_1,\ldots ,A_d)$ to $(B_1,\ldots ,B_d)$, the transposition, and the composition of the conjugation and the transposition. However, since $\overline{R_1(d,n)}=\overline{R(d,n)}$, the conclusion $\varphi (\overline{R(d,n)})=\overline{R(d,n)}$ follows also if $\varphi$ is bijective with polynomial inverse and $\varphi$ and $\varphi ^{-1}$ map $R_1(d,n)$ to its closure. The equivalence in (c) then follows immediately if $\varphi (X_1,X_2,\ldots ,X_d)=(X_1,X_2-p_2(X_1),\ldots ,X_d-p_d(X_1))$.
\end{proof}

\section{Commutative subalgebras of $M_n(\mathbb{F})$ and reducibility of $N(d,n)$ for $d,n\ge 4$}

By Corollary \ref{44} the variety $N(4,4)$ is reducible. In this section we generalize this result. In other words, we show that $N(d,n)$ is reducible for all $d,n\ge 4$. To do so we use the connection between irreducibility of $N(d,n)$ and the dimension of commutative subalgebras of $M_n(\mathbb{F})$ generated by $d$ nilpotent matrices. For commuting matrices $A_1,\ldots ,A_d\in M_n(\mathbb{F})$ let $\mathbb{F}[A_1,\ldots ,A_d]$ denote the unital algebra generated by the matrices $A_1,\ldots ,A_d$. Using the idea from \cite{Gu} we first show that irreducibility of the variety $N(d,n)$ implies that each commutative unital subalgebra of $M_n(\mathbb{F})$ generated by $d$ nilpotent matrices is at most $n$-dimensional. Then for each $n\ge 4$ we construct a $(n+1)$-dimensional commutative unital subalgebra of $M_n(\mathbb{F})$ that is generated by 4 nilpotent matrices, which shows reducibility of $N(d,n)$ for all $d,n\ge 4$. We start with the following proposition whose proof is just extension of the proof of Theorem 1 of \cite{Gu} to an arbitrary number of matrices.

\begin{prop}\label{dim_algebra-closed}
For each positive integer $r$ the set
$$\{(A_1,\ldots ,A_d)\in C(d,n);\dim \mathbb{F}[A_1,\ldots ,A_d]\le r\}$$
is closed in the Zariski topology (i.e. it is an affine variety).
\end{prop}
\begin{proof}
Consider the $n^2\times n^d$ matrix whose columns are the matrices $A_1^{k_1}\cdots A_d^{k_d}$ considered as column vectors of size $n^2$, where $0\le k_i\le n-1$ for $i=1,\ldots ,d$. The rank of this matrix is equal to the dimension of the algebra $\mathbb{F}[A_1,\ldots ,A_d]$. However, the condition that the rank of a matrix is bounded by $r$ is a closed condition, hence the proposition follows.
\end{proof}

\begin{cor}
If the variety $N(d,n)$ is irreducible, then each commutative unital subalgebra of $M_n(\mathbb{F})$ generated by $d$ nilpotent matrices is at most $n$-dimensional.
\end{cor}
\begin{proof}
By the previous proposition the set
$$\mathcal{V}=\{(A_1,\ldots ,A_d)\in N(d,n);\dim \mathbb{F}[A_1,\ldots ,A_d]\le n\}$$
is a subvariety of $N(d,n)$. It contains the set $R_1(d,n)$, since if $A_1$ is 1-regular and $A_2,\ldots ,A_d$ commute with $A_1$, then they are polynomials in $A_1$, and $\mathbb{F}[A_1,\ldots ,A_d]=\mathbb{F}[A_1]$. Therefore it contains also the closure $\overline{R_1(d,n)}$ which is equal to $N(d,n)$, since $N(d,n)$ is by the assumption irreducible.
\end{proof}

Now we can prove reducibility of $N(d,n)$ if $n$ and $d$ are both at least 4.

\begin{thm}
For all positive integers $d,n\ge 4$ the variety $N(d,n)$ is reducible.
\end{thm}
\begin{proof}
By the previous corollary it suffices to find for each $n\ge 4$ some commutative unital subalgebra of $M_n(\mathbb{F})$ of dimension $n+1$ which is generated by 4 nilpotent matrices. Let $A_1$ be the nilpotent matrix in Jordan canonical form which has one Jordan block of order $n-2$ and one Jordan block of order 2, and let $A_2=\mathbf{e_1}\mathbf{e_n}^T$, $A_3=\mathbf{e_{n-1}}\mathbf{e_{n-2}}^T$ and $A_4=\mathbf{e_{n-1}}\mathbf{e_n}^T$ where $\mathbf{e_i}$ denotes the $i$-th standard basis vector of $\mathbb{F}^n$, i.e. the vector with 1 on the $i$-th component and 0 elsewhere. Then the algebra $\mathbb{F}[A_1,A_2,A_3,A_4]$ is generated by 4 nilpotent commuting matrices and spanned by $I,A_1,\ldots, A_1^{n-3},A_2,A_3,A_4$ as a vector space (i.e. $\dim \mathbb{F}[A_1,A_2,A_3,A_4]=n+1$), which proves the theorem.
\end{proof}

\section{Pairs of commuting nilpotent matrices in the centralizer of a nilpotent matrix}

In Section 2 we have seen that $N(3,n)$ is irreducible if and only if it is equal to the Zariski closure of $R(3,n)$. The proof of irreducibility of $N(3,n)$ is usually based on perturbation by triples of 1-regular nilpotent commuting matrices, in the sense that through each triple from $N(3,n)$ we find an irreducible curve (typically a line) in $N(3,n)$ with a Zariski open (and dense) subset contained in $\overline{R(3,n)}$. However, in general the perturbed triples are not easy to find. Much easier to find, when they exist, are commutative nilpotent 1-regular perturbations of only two of the matrices of the triple in the centralizer of the third matrix. (For an example when this is not possible, see Proposition \ref{321} below.) Similarly as in the proofs of irreducibility of $C(3,n)$, it is therefore convenient to study varieties of pairs of commuting nilpotent matrices in the centralizers of given nilpotent matrices. For a nilpotent matrix $A\in M_n(\mathbb{F})$ denote by $C(A)$ the centralizer of $A$, by $N(A)$ the nilpotent centralizer of $A$ (i.e. the set of all {\em nilpotent} $n\times n$ matrices that commute with $A$), and let
$$C_2(A)=\{(B,C)\in C(A)\times C(A);BC=CB\}$$
and
$$N_2(A)=\{(B,C)\in C_2(A);B\, \mathrm{and}\, C\, \mathrm{are}\, \mathrm{nilpotent}\}.$$
Clearly all these sets are affine varieties. The variety $C_2(A)$ is irreducible if $A$ is 3-regular, cf. Corollary 10 of \cite{NS} and Theorem 12 of \cite{S2}. In this section we study analogous problem for $N_2(A)$. In particular we prove that the variety $N_2(A)$ is irreducible if $A$ is 2-regular nilpotent matrix. However, the same does not hold if $A$ is 3-regular. For the proofs of these results we now introduce some additional notation. For a nilpotent $n\times n$ matrix $A$ we define
$$D(A)=\{B\in N(A);\dim \mathbb{F}[A,B]=n\}$$
and
$$D_2(A)=\{(B,C)\in N_2(A);\dim \mathbb{F}[A,B]=n\}.$$
Note that, by Theorem 1.1 of \cite{NSa}, for commuting $n\times n$ matrices $A$ and $B$ the condition $\dim \mathbb{F}[A,B]=n$ is equivalent to $\dim (C(A)\cap C(B))=n$ and to the condition that the algebra $\mathbb{F}[A,B]$ is self-centralizing.

\begin{lemma}
For any nilpotent matrix $A\in M_n(\mathbb{F})$ the sets $D(A)$ and $D_2(A)$ are Zariski open in $N(A)$ and $N_2(A)$, respectively.
\end{lemma}
\begin{proof}
By Proposition \ref{dim_algebra-closed} the set $\{(A',B)\in C(2,n);\dim \mathbb{F}[A',B]\le r\}$ is Zariski closed for each positive integer $r$. On the other hand, by Theorem 1 of \cite{Gu} a unital algebra generated by two commuting $n\times n$ matrices can be at most $n$-dimensional, so the set $\mathcal{V}=\{(A',B)\in C(2,n);\dim \mathbb{F}[A',B]\ne n\}$ is closed, and therefore $(\{A\}\times N(A))\cap \mathcal{V}$ is closed. However, the later set is isomorphic to $N(A)\backslash D(A)$, which shows that $D(A)$ is open in $N(A)$.

For the second part of the lemma observe that
$$N_2(A)\backslash D_2(A)=\Big(\big((N(A)\backslash D(A)\big)\times N(A)\Big)\cap N(2,n),$$
which is closed, so $D_2(A)$ is open in $N_2(A)$.

\end{proof}

The following lemma was proved in \cite{Bas} (Lemma 2.3).

\begin{lemma}[Basili]
For any nilpotent matrix $A$ the variety $N(A)$ is irreducible.
\end{lemma}

\begin{cor}
$\overline{D(A)}=N(A)$ for each nilpotent $n\times n$ matrix $A$.
\end{cor}
\begin{proof}
Because of the previous two lemmas we have to prove only that the set $D(A)$ is nonempty. Since for each invertible matrix $P\in GL_n(\mathbb{F})$ the varieties $N(A)$ and $N(P^{-1}AP)$ are clearly isomorphic, we can assume that the matrix $A$ is in the Jordan canonical form, i.e. $A=\left[
\begin{array}{cccc}J_1\\&J_2\\&&\ddots\\&&&J_k
\end{array}
\right]$ where $J_i$ is the $n_i\times n_i$ nilpotent Jordan block for each $i=1,2,\ldots ,k$, and $n_1\ge n_2\ge \cdots \ge n_k\ge 1$ (and $J_i=0$ in the case of $n_i=1$). Let $B=\left[
\begin{array}{cccc}0&K_1\\&0&\ddots\\&&\ddots&K_{k-1}\\&&&0
\end{array}
\right]$ where  for each $i=1,\ldots ,k-1$, $K_i$ is the $n_i\times n_{i+1}$ matrix of the form $\left[
\begin{array}{cccc}1\\&1\\&&\ddots\\&&&1\\&&&
\end{array}
\right]$. Then the matrix $B$ is nilpotent, it commutes with $A$ and $\dim (\ker A\cap \ker B)=1$, where $\ker X$ denotes the kernel of the matrix $X$. Theorem 2 of \cite{K} then implies that $\dim \mathbb{F}[A,B]=n$, i.e. $B\in D(A)$. Note that, although in \cite{K} it is assumed $\mathbb{F}=\mathbb{C}$, the proof of Theorem 2 of \cite{K} is valid over any field. 
\end{proof}

\begin{prop}\label{D_2irr}
For each nilpotent $n\times n$ matrix $A$ the variety $\overline{D_2(A)}$ is irreducible and of dimension $\dim C(A)-\dim \ker A+n-1$.
\end{prop}
\begin{proof}
Similarly as in the previous corollary we can assume that the matrix $A$ is in the Jordan canonical form. Let $A=J_1\oplus \cdots \oplus J_k$ where for each $i=1,\ldots ,k$ the matrix $J_i$ is the nilpotent Jordan block of size $n_i\times n_i$, and $n_1\ge n_2\ge \cdots \ge n_k\ge 1$. Let $B$ be any nilpotent $n\times n$ matrix that commutes with $A$ such that the algebra generated by $A$ and $B$ is $n$-dimensional. Then the algebra $\mathbb{F}[A,B]$ has the basis
$$\mathcal{B}=\{A^iB^j;0\le j\le k-1,0\le i\le n_{j+1}-1\}$$
(see Theorem 2 of \cite{BH} or Theorem 1.1 of \cite{LL}). We define the polynomial map
$$\varphi \colon D(A)\times \mathbb{F}^{n_1-1}\times \mathbb{F}^{n_2}\times \cdots \times \mathbb{F}^{n_k}\to D_2(A)$$ by
$$\varphi (B,(c_{21},\ldots ,c_{n_11}),(c_{12},\ldots ,c_{n_22})\ldots ,(c_{1k},\ldots ,c_{n_kk}))=$$
$$=(B,\sum _{i=2}^{n_1}c_{i1}A^{i-1}+\sum _{j=2}^k\sum _{i=1}^{n_j}c_{ij}A^{i-1}B^{j-1}).$$
If $(B,C)\in D_2(A)$, then $\dim \mathbb{F}[A,B]=n$, and Theorem 1.1 of \cite{NSa} implies that the algebra $\mathbb{F}[A,B]$ is self-centralizing. In particular, $C$ is a polynomial in $A$ and $B$, and since the matrix $C$ is nilpotent and $\mathcal{B}$ is the basis of $\mathbb{F}[A,B]$, it follows that the map $\varphi$ is surjective. Therefore the closure $\overline{D_2(A)}$ is a closure of a polynomial image of the irreducible variety $N(A)\times \mathbb{F}^{n-1}$, hence it is irreducible. Moreover, since $\mathcal{B}$ is the basis of $\mathbb{F}[A,B]$, the map $\varphi$ is also injective, and $\dim \overline{D_2(A)}=\dim N(A)+n-1$ by the theorem on fibres.

To complete the proof we have to find the dimension of the variety $N(A)$. We can conjugate the matrix $A$, so we can assume that $A=\left[
\begin{array}{cccc}\widetilde{J_1}\\&\widetilde{J_2}\\&&\ddots\\ &&& \widetilde{J_l}
\end{array}
\right]$ for some positive integer $l$ where for each $i=1,\ldots ,l$ the matrix $\widetilde{J_i}=\left[
\begin{array}{cccc}0&I\\&0&\ddots\\&&\ddots&I\\&&&0
\end{array}
\right]$ has $m_i$ block rows and columns of dimension $s_i$, and $m_1>m_2>\cdots >m_l\ge 1$ (where $\widetilde{J_l}=0$ if $m_l=1$). It is clear that a matrix $B$ commutes with $A$ if and only if it is of the form $B=\left[
\begin{array}{cccc}B_{11}&B_{12}&\cdots&B_{1l}\\B_{21}&B_{22}& \cdots&B_{2l}\\\vdots&\vdots&\ddots&\vdots\\B_{l1}& B_{l2}&\cdots&B_{ll}
\end{array}
\right]$ where the matrices $B_{ij}$ are block upper triangular and Toeplitz, i.e.
\begin{itemize}
\item
$B_{ij}=\left[
\begin{array}{cccc}B_{ij}^{(1)}&B_{ij}^{(2)}&\cdots&B_{ij}^{(m_j)} \\0&B_{ij}^{(1)}&\ddots&\vdots\\&\ddots&\ddots& B_{ij}^{(2)}\\\vdots&&\ddots&B_{ij}^{(1)}\\&&&0\\&&& \vdots\\ 0&\cdots&&0
\end{array}
\right]$ for some $B_{ij}^{(1)},\ldots ,B_{ij}^{(m_j)}\in M_{s_i\times s_j}(\mathbb{F})$ if $i<j$,
\item
$B_{ij}=\left[
\begin{array}{ccccccc}0&\cdots&0&B_{ij}^{(1)}&B_{ij}^{(2)}&\cdots& B_{ij}^{(m_i)}\\0&&\cdots&0&B_{ij}^{(1)}&\ddots& \vdots \\\vdots&&&&\ddots&\ddots&B_{ij}^{(2)}\\0&&\cdots&&&0 &B_{ij}^{(1)}
\end{array}
\right]$ for some matrices $B_{ij}^{(1)},\ldots ,B_{ij}^{(m_i)}\in M_{s_i\times s_j}(\mathbb{F})$ if $i>j$, and
\item
$B_{ii}=\left[
\begin{array}{cccc}B_{ii}^{(1)}&B_{ii}^{(2)}&\cdots& B_{ii}^{(m_i)}\\&B_{ii}^{(1)}&\ddots&\vdots\\&&\ddots &B_{ii}^{(2)}\\&&& B_{ii}^{(1)}
\end{array}
\right]$ for some $B_{ii}^{(1)},\ldots ,B_{ii}^{(m_i)}\in M_{s_i}(\mathbb{F})$.
\end{itemize}
By Lemma 2.3 of \cite{Bas} the matrix $B$ belongs to $N(A)$ (i.e. it is nilpotent) if and only if the matrices $B_{11}^{(1)},\ldots ,B_{ll}^{(1)}$ are all nilpotent. Since $\dim N_{s_i}=\dim M_{s_i}(\mathbb{F})-s_i$ for each $i=1,\ldots ,l$, we obtain $\dim N(A)=\dim C(A)-\sum _{i=1}^ls_i=\dim C(A)-\dim \ker A$, which completes the proof of the proposition.
\end{proof}

\begin{cor}
Let $A$ be a nilpotent $n\times n$ matrix. Then the variety $N_2(A)$ is irreducible if and only if $N_2(A)=\overline{D_2(A)}$.
\end{cor}

The next lemma tells us which reduction we can make when proving that a pair from $N_2(A)$ belongs to $\overline{D_2(A)}$.

\begin{lemma}\label{reductionN_2}
Let $(B,C)\in N_2(A)$ and let $p\in \mathbb{F}[t,u]$ be a polynomial in two variables with zero constant term. Then the pair $(B,C)$ belongs to $\overline{D_2(A)}$ if and only if any of the following holds.
\begin{enumerate}
\item[(a)]
$(C,B)\in \overline{D_2(A)}$.
\item[(b)]
$(B,p(A,B)+C)\in \overline{D_2(A)}$.
\end{enumerate}
\end{lemma}
\begin{proof}
If $\varphi \colon N_2(A)\to N_2(A)$ is a polynomial map that maps $D_2(A)$ to $\overline{D_2(A)}$, then $\varphi (\overline{D_2(A)})\subseteq \overline{\varphi (D_2(A))}\subseteq \overline{D_2(A)}$, since polynomial maps are continuous in the Zariski topology. Moreover, if $\varphi$ is bijective and its inverse is also a polynomial map that maps $D_2(A)$ to its closure, then $\varphi (\overline{D_2(A)})=\overline{D_2(A)}$. Now the second equivalence immediately follows, while for the first one we have to prove that the polynomial map $\varphi \colon N_2(A)\to N_2(A)$ defined by $\varphi (B,C)=(C,B)$ maps $D_2(A)$ to $\overline{D_2(A)}$. However, this is clear, since $D_2(A)$ is open subset of $N_2(A)$, and if $(B,C)\in D_2(A)$, then $(C+\lambda B,B)\in D_2(A)$ for all except finitely many scalars $\lambda \in \mathbb{F}$, and therefore $(C,B)\in \overline{D_2(A)}$.
\end{proof}

We can now prove the main theorem of this section. Note that Corollary 10 of \cite{NS} is the analogous result for $C_2(A)$ instead of $N_2(A)$.

\begin{thm}\label{2-regN_2irr}
For any 2-regular nilpotent $n\times n$ matrix $A$ the variety $N_2(A)$ is irreducible.
\end{thm}
\begin{proof}
Let $(B,C)\in N_2(A)$ be an arbitrary pair. We have to prove that $(B,C)\in \overline{D_2(A)}$. If $C$ is a polynomial in $A$ and $B$ or if $B$ is a polynomial in $A$ and $C$, then by Lemma \ref{reductionN_2}(a) we can assume the first case, and moreover, by Lemma \ref{reductionN_2}(b) we can assume that $C=0$. Since the set $D(A)$ is nonempty, there exists a nilpotent matrix $X\in C(A)$ such that $\dim \mathbb{F}[A,X]=n$, i.e. $(X,0)\in D_2(A)$. The line $\{(\lambda X+(1-\lambda )B,0);\lambda \in \mathbb{F}\}$ then intersects the open subset $D_2(A)$ of $N_2(A)$, therefore $(\lambda X+(1-\lambda )B,0)\in D_2(A)$ for all except finitely many scalars $\lambda \in \mathbb{F}$. However, then the whole line belongs to $\overline{D_2(A)}$, and in particular $(B,0)\in \overline{D_2(A)}$.

In the sequel we assume that $C$ is not a polynomial in $A$ and $B$ and also that $B$ is not a polynomial in $A$ and $C$. For any invertible matrix $P\in GL_n(\mathbb{F})$ the varieties $N_2(A)$ and $N_2(P^{-1}AP)$ are clearly isomorphic, therefore we can assume that the matrix $A$ is in the Jordan canonical form, i.e. $A=\left[
\begin{array}{cc}J_k&0\\0&J_m
\end{array}
\right]$ where $J_k$ and $J_m$ are Jordan blocks of orders $k$ and $m$, respectively, and $k\ge m$. We will consider two cases.

{\bf Case 1.} Assume that $k>m$. By Lemma 4 of \cite{NS} the elements of the centralizer $C(A)$ can be identified with the matrices of the form $\left[
\begin{array}{cc}p(t)&t^{k-m}q(t)\\r(t)&s(t)
\end{array}
\right]$ where the first row belongs to $\mathbb{F}[t]/t^k$ and the second row belongs to $\mathbb{F}[t]/t^m$, i.e. $p,q\in \mathbb{F}[t]/t^k$, $r,s\in \mathbb{F}[t]/t^m$ and $q$ is of degree less than $m$. Moreover, as observed in the remark after Lemma 4 of \cite{NS} the multiplication in $C(A)$ corresponds to the multiplication of such polynomial matrices. In Case 1 we will use this polynomial notation for all elements of $C(A)$. Note that in this notation $A$ corresponds to the matrix $\left[
\begin{array}{cc}t&0\\0&t
\end{array}
\right]$. The matrices $B$ and $C$ commute with $A$, they are nilpotent and by Lemma \ref{reductionN_2} we can add any polynomials (with zero constant terms) in $A$ to them, therefore we can assume that $B=\left[
\begin{array}{cc}0&t^{k-m}p(t)\\q(t)&r(t)
\end{array}
\right]$ and $C=\left[
\begin{array}{cc}0&t^{k-m}p'(t)\\q'(t)&r'(t)
\end{array}
\right]$ for some polynomials $p,p',q,q',r,r'\in \mathbb{F}[t]/t^m$. Since $B$ and $C$ are nilpotent, the polynomials $r$ and $r'$ have zero constant terms.

{\bf (1)} Assume first that at least one of the polynomials $p,p',q,q'$ has nonzero constant term. By Lemma \ref{reductionN_2}(a) we can assume that $p$ or $q$ has nonzero constant term. Then it is easy to see that in the first case $\dim (\ker A\cap \ker B)=1$ and in the second case $\dim (\ker A^T\cap \ker B^T)=1$. Theorem 2 of \cite{K} then in both cases implies that $\dim \mathbb{F}[A,B]=n$, therefore $(B,C)\in D_2(A)$.

{\bf (2)} Assume now that the constant terms of the polynomials $p,p',q,q'$ are all zero. Let $\beta$ be the largest positive integer such that $t^{\beta}$ divides $p(t)$, $q(t)$ and $r(t)$, and let $\gamma$ be the largest positive integer such that $t^{\gamma}$ divides $p'(t)$, $q'(t)$ and $r'(t)$. Then $B=A^{\beta}\left[
\begin{array}{cc}0&t^{k-m}\tilde{p}(t)\\\tilde{q}(t)&\tilde{r}(t)
\end{array}
\right]$ and $C=A^{\gamma}\left[
\begin{array}{cc}0&t^{k-m}\tilde{p'}(t)\\\tilde{q'}(t)&\tilde{r'}(t)
\end{array}
\right]$, where at least one of the polynomials $\tilde{p}(t),\tilde{q}(t),\tilde{r}(t)$ is not divisible by $t$ and at least one of $\tilde{p'}(t),\tilde{q'}(t),\tilde{r'}(t)$ is not divisible by $t$. If $t$ divides $\tilde{r}(t)$ or $\tilde{r'}(t)$, then by Lemma \ref{reductionN_2}(a) we can assume that $t$ divides $\tilde{r}(t)$. The matrix $X=\left[
\begin{array}{cc}0&t^{k-m}\tilde{p}(t)\\\tilde{q}(t)&\tilde{r}(t)
\end{array}
\right]$ is then nilpotent and it commutes with $B$. Since at least one of the polynomials $\tilde{p}(t),\tilde{q}(t)$ is not divisible by $t$, for each $\lambda \ne 0$ the pair $(B,C+\lambda X)$ belongs to $\overline{D_2(A)}$ by (1). Therefore $(B,C)\in \overline{D_2(A)}$.

However, if $\tilde{r}(t)$ and $\tilde{r'}(t)$ are not divisible by $t$, then by Lemma \ref{reductionN_2}(a) we can assume that $\beta \le \gamma$. Moreover, the matrix $\tilde{r}(A)$ is invertible and its inverse is a polynomial in $A$, therefore by Lemma \ref{reductionN_2}(b) the pair $(B,C)$ belongs to $\overline{D_2(A)}$ if and only if the pair $(B,C-A^{\gamma -\beta}\tilde{r'}(A)\tilde{r}(A)^{-1}B)$ belongs to $\overline{D_2(A)}$, and we can assume that $\tilde{r'}(t)=0$. Since the matrix $C$ is not a polynomial in $A$ and $B$, one of the polynomials $\tilde{q'}$ and $\tilde{p'}$ is nonzero, therefore the pair $(B,C)$ belongs to $\overline{D_2(A)}$ by the previous paragraph.

{\bf Case 2.} Assume that $k=m$. Since for any invertible matrix $P\in GL_n(\mathbb{F})$ the varieties $N_2(A)$ and $N_2(P^{-1}AP)$ are isomorphic, we can assume that the matrix $A$ is of the block form $A=\left[
\begin{array}{cccc}0&I&\cdots&0\\&\ddots&\ddots&\vdots\\&&\ddots&I \\&&&0
\end{array}
\right]$ with $k$ rows and $k$ columns, and each block is a $2\times 2$ matrix. Then $B=\left[
\begin{array}{cccc}B_1&B_2&\cdots&B_k\\&\ddots&\ddots&\vdots\\&& \ddots&B_2\\&&&B_1
\end{array}
\right]$ and $C=\left[
\begin{array}{cccc}C_1&C_2&\cdots&C_k\\&\ddots&\ddots&\vdots\\&& \ddots&C_2\\&&&C_1
\end{array}
\right]$ for some $2\times 2$ matrices $B_i$ and $C_i$. Clearly $B_1$ and $C_1$ have to be nilpotent matrices. By the remark after Lemma 4 of \cite{NS} these matrices $B$ and $C$ can be identified with the polynomials $B(t)=B_1+B_2t+\cdots +B_kt^{k-1}$ and $C(t)=C_1+C_2t+\cdots +C_kt^{k-1}$ in $M_2(\mathbb{F})[t]/t^k$. We will use this notation in Case 2.

{\bf (1)} If $B_1$ or $C_1$ is nonzero matrix, then by Lemma \ref{reductionN_2}(a) we can assume that $B_1\ne 0$. However, then $\dim (\ker A\cap \ker B)=1$, which implies $\dim \mathbb{F}[A,B]=n$ by Theorem 2 of \cite{K}, therefore $(B,C)\in D_2(A)$.

{\bf (2)} Assume now that $B_1=C_1=0$. If $B_2$ is not a scalar matrix, then the commutativity relation of $B$ and $C$ implies that $B_2C_2=C_2B_2$ which is equivalent to $C_2=\alpha I+\beta B_2$ for some $\alpha ,\beta \in \mathbb{F}$. Moreover, by Lemma \ref{reductionN_2}(b) we can assume that $C_2=0$. Then the commutativity relation implies that $B_2C_3=C_3B_2$, which again means that $C_3$ is linearly dependent of $I$ and $B_2$, and using Lemma \ref{reductionN_2}(b) we can assume that $C_3=0$. Repeating this we can assume that $C_i=0$ for each $i\le k-1$. Since the variety of all $2\times 2$ matrices that are not generic is 3-dimensional, its intersection with any 2-dimensional vector subspace of $M_2(\mathbb{F})$ is at least 1-dimensional by Theorem 6 in Chapter I, \S 6.2 of \cite{Sh}. In particular, there exist $\mu ,\nu \in \mathbb{F}$, at least one of them nonzero, such that the matrix $\mu B_2-\nu C_k$ is not generic. If $\nu =0$, then $B_2$ is not generic, and since we can add any multiple of $A$ to $B$, we can assume that $B_2$ is nilpotent. The matrix $\widetilde{B}(t)=B_2+B_3t+\cdots +B_kt^{k-2}$ is then nilpotent and it commutes with $B$. Since $(B(t),C(t)+\lambda \widetilde{B}(t))\in \overline{D_2(A)}$ for all $\lambda \ne 0$ by (1), we obtain $(B,C)\in \overline{D_2(A)}$. However, if $\nu \ne 0$, then by Lemma \ref{reductionN_2}(b) we can subtract $\frac{\mu}{\nu}A^{k-2}B$ from $C$, therefore we can assume that $C_k$ is not generic, and by the same lemma we can assume that it is nilpotent. Since $C$ is not a polynomial in $A$ and $B$, its submatrix $C_k$ is nonzero. The matrix $\widetilde{C}(t)=C_k\in C(A)$ is then nilpotent, it commutes with $C$ and by (1) the pair $(B(t)+\lambda \widetilde{C}(t),C(t))$ belongs to $D_2(A)$ for each $\lambda \ne 0$. Therefore $(B,C)\in \overline{D_2(A)}$.

It remains to consider the case when $B_2$ is a scalar matrix. Since $C$ is not a polynomial in $A$, there exists the smallest positive integer $l$ such that $C_l$ is not a scalar matrix. The matrix $\widetilde{C}(t)=C_lt+\cdots +C_kt^{k-l+1}$ then commutes with $C$, and by the previous paragraph $(B(t)+\lambda \widetilde{C}(t),C(t))\in \overline{D_2(A)}$ for each $\lambda \ne 0$. However, then $(B,C)\in \overline{D_2(A)}$, which completes the proof of the theorem.
\end{proof}

\begin{cor}\label{2-regN(3,n)irr}
If $(A,B,C)$ is any triple of commuting nilpotent $n\times n$ matrices with $A$ 2-regular, then $(A,B,C)\in \overline{R(3,n)}$.
\end{cor}
\begin{proof}
Let $A$ be 2-regular and $(X,Y)\in D_2(A)$ arbitrary pair. Then the pair $(X,0)$ also belongs to $D_2(A)$ and $Y$ is a polynomial in $A$ and $X$ by Theorem 1.1 of \cite{NSa}. By Theorem 3.7 of \cite{P} the variety $N(2,n)$ is irreducible for all positive integers $n$, therefore $(A,X)\in \overline{R(2,n)}$ and $(A,X,0)\in \overline{R(2,n)\times \{0\}}$. However, since $R(2,n)\times \{0\}\subseteq R(3,n)$, it follows that $(A,X,0)\in \overline{R(3,n)}$. Now let $p\in \mathbb{F}[t,u]$ be a polynomial such that $Y=p(A,X)$. The polynomial map $\varphi \colon N(3,n)\to N(3,n)$ defined by $\varphi (A_1,A_2,A_3)=(A_1,A_2,A_3+p(A_1,A_2))$ maps $R_1(3,n)$ to itself, therefore $\varphi (\overline{R(3,n)})\subseteq \overline{R(3,n)}$, and in particular $(A,X,Y)=\varphi (A,X,0)\in \overline{R(3,n)}$. We proved that $\{A\}\times D_2(A)\subseteq \overline{R(3,n)}$, therefore $\{A\}\times \overline{D_2(A)}\subseteq \overline{R(3,n)}$, and in particular $(A,B,C)\in \overline{R(3,n)}$, since $N_2(A)$ is irreducible.
\end{proof}

Theorem \ref{2-regN_2irr} can be viewed as a nilpotent version of Corollary 10 of \cite{NS}. However, the same result for 3-regular matrices, that would be the nilpotent version of Theorem 12 of \cite{S2}, does not hold as is shown in the following proposition.

\begin{prop}\label{N_2red}
Let $A=\left[
\begin{array}{cccccc}0&1&0&0&0&0\\0&0&1&0&0&0\\0&0&0&0&0&0\\ 0&0&0&0&1&0\\0&0&0&0&0&0\\0&0&0&0&0&0
\end{array}
\right]$. Then $A$ is 3-regular nilpotent matrix and the variety $N_2(A)$ is reducible.
\end{prop}
\begin{proof}
Let $\mathcal{V}$ be the variety of all pairs of $6\times 6$ matrices
$$\left(\left[
\begin{array}{cccccc}0&a&b&c&d&e\\0&0&a&0&c&0\\0&0&0&0&0&0\\ 0&0&f&0&g&0\\0&0&0&0&0&0\\0&0&h&0&i&0
\end{array}
\right] ,\left[
\begin{array}{cccccc}0&a'&b'&c'&d'&e'\\0&0&a'&0&c'&0\\0&0&0&0&0&0 \\0&0&f'&0&g'&0\\0&0&0&0&0&0\\0&0&h'&0&i'&0
\end{array}
\right]\right)$$
satisfying $cf'+eh'=c'f+e'h$ and $ac'+cg'+ei'=a'c+c'g+e'i$. Clearly $\mathcal{V}$ is a proper subvariety of $N_2(A)$, and its dimension is at least 16, since it is a subset of $\mathbb{F}^{18}$ defined by 2 equations. However, by Proposition \ref{D_2irr} the closure $\overline{D_2(A)}$ has dimension equal to $\dim C(A)-\dim \ker A+5=16$, therefore $\overline{D_2(A)}$ cannot be equal to $N_2(A)$, and $N_2(A)$ is reducible.
\end{proof}

\begin{rmk}
As we will see in the next section, although the variety $N_2(A)$ in the previous proposition is reducible, the variety $N(3,6)$ is still irreducible. However, to prove this one needs to have all three matrices of the triples from $\{A\}\times N_2(A)$ perturbed by 1-regular nilpotent commuting matrices, not only two of them.
\end{rmk}

\section{Irreducibility of varieties $N(3,n)$ for $n\le 6$}

In the previous section we proved that each triple of commuting nilpotent $n\times n$ matrices that contains a 2-regular matrix belongs to $\overline{R(3,n)}$. In this section we consider all other cases that occur in dimensions 5 and 6. Eventually, we prove that the varieties $N(3,5)$ and $N(3,6)$ are irreducible. We start with two results that hold in any dimension. We prove that a triple of commuting nilpotent $n\times n$ matrices belongs to $\overline{R(3,n)}$ if the Jordan canonical forms of the matrices in the triple have only one nonzero Jordan block or if they have Jordan blocks of orders at most 2. These results will be proved under some additional assumptions on the perturbability of triples of matrices of higher rank, which, however, have no influence on irreducibility of $N(3,n)$, since in the case of irreducibility of $N(3,n)$ all triples have to belong to $\overline{R(3,n)}$. Note that these results are nilpotent versions of Theorems 5.3 and 5.1 of \cite{HO}.

\begin{prop}\label{1nonzero}
Let $k<n$ and assume that each triple of commuting nilpotent $n\times n$ matrices whose linear span contains a matrix of rank at least $k$ belongs to $\overline{R(3,n)}$. Let $(A,B,C)\in N(3,n)$ be a triple such that the Jordan canonical form of $A$ has one Jordan block of order $k$ and $n-k$ zero Jordan blocks. Then $(A,B,C)\in \overline{R(3,n)}$.
\end{prop}
\begin{proof}
If $k=n-1$, then the matrix $A$ is 2-regular and the proposition follows from Corollary \ref{2-regN(3,n)irr}. Therefore we assume that $k\le n-2$. By Lemma \ref{reduction}(a) we can assume that the matrix $A$ is in the Jordan canonical form. If $J$ denotes the $k\times k$ Jordan block, then $B=\left[
\begin{array}{cc}p(J)&\mathbf{e_1}\mathbf{a}^T\\\mathbf{b} \mathbf{e_k}^T&D
\end{array}
\right]$ and $C=\left[
\begin{array}{cc}p'(J)&\mathbf{e_1}\mathbf{a'}^T\\ \mathbf{b'}\mathbf{e_k}^T&D'
\end{array}
\right]$ for some polynomials with zero constant terms $p,p'\in \mathbb{F}[t]$, some vectors $\mathbf{a},\mathbf{a'},\mathbf{b},\mathbf{b'}\in \mathbb{F}^{n-k}$ and some matrices $D,D'\in M_{n-k}(\mathbb{F})$. By Lemma \ref{reduction}(c) we can assume that $p$ and $p'$ are zero polynomials. Moreover, if $D$ or $D'$ is nonzero, then there exists some linear combination of the matrices $A$, $B$ and $C$ which is of rank at least $k$, and the triple $(A,B,C)$ belongs to $\overline{R(3,n)}$ by the assumption of the proposition. Therefore we can assume that $D=D'=0$. Then the triple $(A,B,C)$ belongs to a subvariety of $N(3,n)$ which is isomorphic to $\{(\mathbf{a},\mathbf{a'},\mathbf{b},\mathbf{b'})\in \mathbb{F}^{4(n-k)};\mathbf{a}^T\mathbf{b'}=\mathbf{a'}^T\mathbf{b}\}$. This variety is irreducible, since it is defined by one irreducible polynomial. Therefore by Lemma \ref{irr-open} we can assume any open condition on the matrices $B$ and $C$. We assume that the vectors $\mathbf{a}$ and $\mathbf{a'}$ are linearly independent and that $\mathbf{a}^T\mathbf{b}\ne 0$ and $\mathbf{a'}^T\mathbf{b'}\ne 0$. In particular, the vector $\mathbf{a}$ is nonzero and by Lemma \ref{reduction}(a) we can assume that $\mathbf{a}=\mathbf{e_1}$. Let $\beta =\mathbf{e_1}^T\mathbf{b}$ and $Q=\left[
\begin{array}{cc}1&0\\\mathbf{\tilde{b}}&I
\end{array}
\right] \in GL_{n-k}(\mathbb{F})$ where $\mathbf{\tilde{b}}$ is the vector of the last $n-k-1$ components of $\frac{1}{\beta}\mathbf{b}$. Then $\mathbf{e_1}^TQ=\mathbf{e_1}^T$ and $\beta Q\mathbf{e_1}=\mathbf{b}$. Let $P=\left[
\begin{array}{cc}I&0\\0&Q
\end{array}
\right] \in C(A)$. By Lemma \ref{reduction}(a) the triple $(A,B,C)$ belongs to $\overline{R(3,n)}$ if and only if the triple $(A,P^{-1}BP,P^{-1}CP)$ does, therefore we can assume that $\mathbf{b}=\beta \mathbf{e_1}$. Moreover, by Lemma \ref{reduction} we can add any multiple of $B$ to $C$ and change the basis of $\mathbb{F}^n$, therefore we can assume that $\mathbf{a'}=\mathbf{e_2}$. The commutativity relation of $B$ and $C$ then implies that $\mathbf{e_1}^T\mathbf{b'}=0$, and similarly as we assumed $\mathbf{b}=\beta \mathbf{e_1}$ we can assume also that $\mathbf{b'}=-\gamma \mathbf{e_2}$ for some nonzero $\gamma \in \mathbb{F}$. We define $(n-k)\times (n-k)$ matrices  $X=\left[
\begin{array}{ccc}\beta\sqrt{\gamma}&\beta\sqrt{\beta}&0\\-\gamma \sqrt{\beta}&-\beta\sqrt{\gamma}&0\\0&0&0
\end{array}
\right]$ and $X'=\left[
\begin{array}{ccc}-\gamma\sqrt{\beta}&-\beta\sqrt{\gamma}&0\\ \gamma\sqrt{\gamma}&\gamma\sqrt{\beta}&0\\0&0&0
\end{array}
\right]$, where the first two rows and columns are of dimension 1 and the last ones are of dimension $n-k-2$ (possibly 0), and let $Y=\left[
\begin{array}{cc}0&0\\0&X
\end{array}
\right]$ and $Z=\left[
\begin{array}{cc}0&0\\0&X'
\end{array}
\right]$, where the first rows and columns are of dimension $k$ and the last ones of dimension $n-k$. The matrices $X$ and $X'$ are of square zero, they commute, and $\mathbf{e_1}^TX'=\mathbf{e_2}^TX$ and $\beta X'\mathbf{e_1}=-\gamma X\mathbf{e_2}$, therefore $(A,B+\lambda Y,C+\lambda Z)\in N(3,n)$ for each $\lambda \in \mathbb{F}$. Since the matrix $X$ is nonzero, for each $\lambda \ne 0$ some linear combination of $A$ and $B+\lambda Y$ has rank $k$, therefore $(A,B+\lambda Y,C+\lambda Z)\in \overline{R(3,n)}$ by the assumption of the proposition. However, then the triple $(A,B,C)$ also belongs to $\overline{R(3,n)}$.
\end{proof}

Now we deal with triples whose matrices generate a vector space of square zero matrices. We first prove a lemma, which can be seen as the nilpotent version of Proposition 4.3 of \cite{HO}. Our proof is just a slight modification of that of Theorem 1 in \cite{O}.

\begin{lemma}
Let $l$ be a positive integer, $m\ge 2$ and let $W\in M_l(\mathbb{F})$ and $V\in M_{l\times m}(\mathbb{F})$ be arbitrary. Let
$$A=\left[
\begin{array}{ccc}0&I&0\\0&0&0\\0&0&0
\end{array}
\right]\quad \mathrm{and}\quad B=\left[
\begin{array}{ccc}0&W&V\\0&0&0\\0&0&0
\end{array}
\right]$$
where the first two rows and columns are of dimension $l$ and the last ones are of dimension $m$. Then there exists a nonzero nilpotent matrix of the form $N=\left[
\begin{array}{ccc}N_1&0&0\\0&N_1&0\\0&N_2&N_3
\end{array}
\right]$  which commutes with $A$ and with $B$.
\end{lemma}
\begin{proof}
If the matrix $V$ is not injective, then there exists a nonzero matrix $X\in M_{m\times l}(\mathbb{F})$ such that $VX=0$. Then the matrix $N=\left[
\begin{array}{ccc}0&0&0\\0&0&0\\0&X&0
\end{array}
\right]$ is nilpotent and it commutes with $A$ and with $B$.

In the sequel we assume that $V$ is injective, and in particular $l\ge m$. The conclusion of the lemma clearly remains the same if we conjugate the matrix $B$ by any invertible matrix of the form $P=\left[
\begin{array}{ccc}P_1&0&0\\0&P_1&0\\0&P_2&P_3
\end{array}
\right]$ (which belongs to the centralizer of $A$). In particular, since $V$ is injective, there exist $R\in M_{m\times l}(\mathbb{F})$ and $Q\in GL_l(\mathbb{F})$ such that $[Q^{-1}WQ+Q^{-1}VR\quad Q^{-1}V]=\left[
\begin{array}{ccc}W_1'&W_2'&0\\0&0&I
\end{array}
\right]$ for some matrices $W_1'\in M_{l-m}(\mathbb{F})$ and $W_2'\in M_{(l-m)\times m}(\mathbb{F})$. Since we can conjugate the matrix $B$ by the matrix $P=\left[
\begin{array}{ccc}Q&0&0\\0&Q&0\\0&R&I
\end{array}
\right]$, we can assume that $[W\quad V]=\left[
\begin{array}{ccc}W_1'&W_2'&0\\0&0&I
\end{array}
\right]$ where the first row and column are of dimension $l-m$ and the others are of dimension $m$. Assume inductively that, for some positive integer $t$, $[W\quad V]=\left[
\begin{array}{ccccc}W_1&W_2\\&&I\\&&&\ddots\\&&&&I
\end{array}
\right]$ where the first row and column are of dimension $l-tm$ and the other $t$ rows and $t+1$ columns are of dimension $m$.

Assume first that $l-tm\ne 0$. If $W_2$ is not injective, then (because of $m\ge 2$) there exists a nonzero nilpotent matrix $N'\in M_m(\mathbb{F})$ such that $W_2N'=0$. We define $N''=\left[
\begin{array}{ccccc}0\\&N'\\&&N'\\&&&\ddots\\&&&&N'
\end{array}
\right]$ where the first row and column are of dimension $l-tm$ and the other $t$ rows and columns are of dimension $m$. The matrix  $N=\left[
\begin{array}{ccc}N''&0&0\\0&N''&0\\0&0&N'
\end{array}
\right]$ is then nilpotent and it commutes with $A$ and with $B$. On the other hand, if $W_2$ is injective, then $l-tm$ is not smaller than $m$ and there exist $R\in M_{m\times (l-tm)}(\mathbb{F})$ and $S \in GL_{l-tm}(\mathbb{F})$ such that $[S^{-1}W_1S+S^{-1}W_2R\quad S^{-1}W_2]=\left[
\begin{array}{ccc}\widetilde{W_1}&\widetilde{W_2}&0\\0&0&I
\end{array}
\right]$ for some matrices $\widetilde{W_1}\in M_{l-(t+1)m}(\mathbb{F})$ and $\widetilde{W_2}\in M_{(l-(t+1)m)\times m}(\mathbb{F})$. Let $\widehat{W_1}=\left[
\begin{array}{cc}\widetilde{W_1}&\widetilde{W_2}\\0&0
\end{array}
\right] \in M_{l-tm}(\mathbb{F})$, $\widehat{W_2}=\left[
\begin{array}{c}0\\I
\end{array}
\right] \in M_{(l-tm)\times m}(\mathbb{F})$,
$$Q=\left[
\begin{array}{cccccc}S\\R&I\\R\widehat{W_1}&R\widehat{W_2}&I\\R \widehat{W_1}^2&R\widehat{W_1}\widehat{W_2}&R \widehat{W_2}&I\\\vdots&\vdots&\vdots&\ddots&\ddots\\ R\widehat{W_1}^{t-1}&R\widehat{W_1}^{t-2} \widehat{W_2}&R\widehat{W_1}^{t-3}\widehat{W_2}& \cdots&R\widehat{W_2}&I
\end{array}
\right] \in M_l(\mathbb{F})$$
and $T=\left[
\begin{array}{ccccc}R\widehat{W_1}^t&R\widehat{W_1}^{t-1} \widehat{W_2}&R\widehat{W_1}^{t-2}\widehat{W_2}& \cdots&R\widehat{W_2}
\end{array}
\right] \in M_{m\times l}(\mathbb{F})$. Since we can conjugate the matrix $B$ by the matrix $P=\left[
\begin{array}{ccc}Q&0&0\\0&Q&0\\0&T&I
\end{array}
\right]$, we can assume that $[W\quad V]=\left[
\begin{array}{ccccc}\widetilde{W_1}&\widetilde{W_2}\\&&I\\&&& \ddots \\&&&&I
\end{array}
\right]$, where the first row and column are of dimension $l-(t+1)m$ and the other $t+1$ rows and $t+2$ columns are of dimension $m$, and we can proceed with the induction.

However, if $l-tm=0$, then because of $m\ge 2$ there exists a nonzero nilpotent matrix $N'\in M_m(\mathbb{F})$. We define $N''=\left[
\begin{array}{ccc}N'\\&\ddots\\&&N'
\end{array}
\right] \in M_l(\mathbb{F})$ and $N=\left[
\begin{array}{ccc}N''&0&0\\0&N''&0\\0&0&N'
\end{array}
\right]$. Then $N$ is a nilpotent matrix that commutes with $A$ and with $B$, which completes the proof of the lemma.
\end{proof}

\begin{rmk}
Note that the above proof would not work for $m=1$, even if we assume that $\mathrm{rank}\, B\le l-1$, which was assumed in \cite{HO} and \cite{O}. As we will see in the proof of the next proposition, the lemma still holds if $m=1$ and $\mathrm{rank}\, B\le l-1$, but it has to be proved slightly differently.
\end{rmk}

\begin{prop}\label{square-zero}
Let $l$ be a positive integer, $m$ a nonnegative integer and $n=2l+m$. Assume that each triple of commuting nilpotent $n\times n$ matrices whose linear span contains either a matrix of rank at least $l+1$ or a matrix of rank $l$ with nonzero square belongs to $\overline{R(3,n)}$. Let $(A,B,C)\in N(3,n)$ be a triple such that the Jordan canonical form of $A$ has $l$ Jordan blocks of order 2 and $m$ zero Jordan blocks. Then $(A,B,C)\in \overline{R(3,n)}$.
\end{prop}
\begin{proof}
By Lemma \ref{reduction}(a) we can conjugate the matrices by any invertible $n\times n$ matrix, therefore we can assume that $A=\left[
\begin{array}{ccc}0&I&0\\0&0&0\\0&0&0
\end{array}
\right]$, where the first two block rows and columns are of dimension $l$ and the last ones are of dimension $m$ (possibly 0). Then $B=\left[
\begin{array}{ccc}D&E&F\\0&D&0\\0&G&H
\end{array}
\right]$ and $C=\left[
\begin{array}{ccc}D'&E'&F'\\0&D'&0\\0&G'&H'
\end{array}
\right]$ for some matrices $D,D',E,E'\in M_l(\mathbb{F})$, $F,F'\in M_{l\times m}(\mathbb{F})$, $G,G'\in M_{m\times l}(\mathbb{F})$ and $H,H'\in M_m(\mathbb{F})$. By the assumption of the proposition we can assume that each linear combination of the matrices $A$, $B$ and $C$ is a square-zero matrix of rank at most $l$. The second condition immediately implies that $H=H'=0$ if $m>0$.

If $D\ne 0$ or $D'\ne 0$, then by Lemma \ref{reduction}(b) we can assume that $D\ne 0$ and that either the matrices $D$ and $D'$ are linearly independent or $D'=0$. The coefficient at $\lambda$ in the polynomial $(B+\lambda A)^2$ is equal to $\left[
\begin{array}{ccc}0&2D&0\\0&0&0\\0&0&0
\end{array}
\right]$. If $\mathrm{char}\, \mathbb{F}\ne 2$, then there exists $\lambda \in \mathbb{F}$ such that $(B+\lambda A)^2\ne 0$, and therefore $(A,B,C)\in \overline{R(3,n)}$ by the assumption of the proposition. Therefore we can assume that $\mathrm{char}\, \mathbb{F}=2$. The condition $B^2=0$ is equivalent to the equations $D^2=0$, $DE+ED+FG=0$, $DF=0$ and $GD=0$. In particular, $D$ and $D'$ are square-zero matrices, and the condition that either $D$ and $D'$ are linearly independent or $D'=0$ implies that the algebra $\mathbb{F}[D]=\mathbb{F}\cdot I+\mathbb{F}\cdot D$ is not contained in the algebra $\mathbb{F}[D']=\mathbb{F}\cdot I+\mathbb{F}\cdot D'$. Recall that for any matrix $X\in M_l(\mathbb{F})$ we denoted by $C(X)$ the centralizer of $X$. Similarly we denote by $C(C(X))$ the centralizer of the algebra $C(X)$ in $M_l(\mathbb{F})$ (i.e. the set of all matrices from $M_l(\mathbb{F})$ that commute with all elements of $C(X)$) and by $Z(C(X))$ the center of the algebra $C(X)$ (i.e. $Z(C(X))=C(C(X))\cap C(X)$). Theorem 7 in Chapter 1 of \cite{ST} yields $Z(C(X))=\mathbb{F}[X]$ for all $X\in M_l(\mathbb{F})$, and since each matrix that commutes with all elements of $C(X)$ clearly belongs to $C(X)$, we obtain $C(C(X))=\mathbb{F}[X]$ for all $X\in M_l(\mathbb{F})$. Since the algebra $\mathbb{F}[D]$ is not contained in $\mathbb{F}[D']$, the centralizer $C(D')$ is not contained in $C(D)$. Therefore there exists a matrix $Y\in C(D')\backslash C(D)$. Note that then $DY+YD\ne 0$, since $\mathrm{char}\, \mathbb{F}=2$. The matrix $X=\left[
\begin{array}{ccc}0&Y&0\\0&0&0\\0&0&0
\end{array}
\right]$ then commutes with $A$ and $C$, and $(B+\lambda X)^2\ne 0$ for $\lambda \ne 0$. The assumption of the proposition then implies that $(A,B+\lambda X,C)\in \overline{R(3,n)}$ for all $\lambda \ne 0$, hence $(A,B,C)\in \overline{R(3,n)}$. In the rest of the proof we can therefore assume that $D=D'=0$.

Assume first that $m=0$. Since the variety of all $l\times l$ matrices that are not generic is $(l^2-1)$-dimensional, its intersection with each 2-dimensional vector subspace of $M_l(\mathbb{F})$ is at least 1-dimensional by Theorem 6 in Chapter I, \S 6.2 of \cite{Sh}. In particular, there exist $\lambda ,\mu \in \mathbb{F}$, not both of them zero, such that the matrix $\lambda E+\mu E'$ is not generic. Moreover, by Lemma \ref{reduction}(b) we can assume that $E$ is not generic. Therefore it commutes with some nonzero nilpotent matrix $N'\in M_l(\mathbb{F})$. The matrix $N=\left[
\begin{array}{cc}N'&0\\0&N'
\end{array}
\right]$ is then nilpotent and it commutes with $B$, so $(A,B,C+\lambda N)\in N(3,n)$ for each $\lambda \in \mathbb{F}$. Since $N'\ne 0$, the triple $(A,B,C+\lambda N)$ belongs to $\overline{R(3,n)}$ for all $\lambda \ne 0$ by the previous paragraph, therefore $(A,B,C)\in \overline{R(3,n)}$.

In the rest of the proof we will assume that $m>0$. Since $B^2=0$, it follows that $FG=0$. Moreover, for each $\lambda\in \mathbb{F}$ the rank of the matrix $\lambda A+B$ is at most $l$, which implies that $\mathrm{rank}\, \left[
\begin{array}{cc}\lambda I+E&F\\G&0
\end{array}
\right] \le l$. In particular, for each $i,j=1,\ldots ,m$, $\det \left[
\begin{array}{cc}\lambda I+E&F\mathbf{e_i}\\\mathbf{e_j}^TG&0
\end{array}
\right] =0$. The coefficient at $\lambda ^{l-1}$ in this determinant is equal to
$$\sum _{k=1}^l\mathbf{e_j}^TG\mathbf{e_k}\mathbf{e_k}^TF \mathbf{e_i}=\mathbf{e_j}^TGF\mathbf{e_i},$$
therefore $\mathbf{e_j}^TGF\mathbf{e_i}=0$ for all $i,j=1,\ldots ,m$, i.e. $GF=0$.

Assume first that there exist $\mu ,\nu \in \mathbb{F}$ such that the matrices $\mu G+\nu G'$ and $\mu F+\nu F'$ are both nonzero. By Lemma \ref{reduction}(b) we can then assume that $F$ and $G$ are nonzero matrices. Since $GF=0$ and $FG=0$, it follows that $m>1$ and that the rank of the matrices $F$ and $G$ is strictly smaller than $\min \{m,l\}$. By Lemma \ref{reduction}(a) we can conjugate the matrices $B$ and $C$ by any invertible matrix in the centralizer of $A$, therefore we can assume that $F=\sum _{i=1}^k\mathbf{e_i}\mathbf{e_i}^T$ for some $k\le \min \{m,l\}-1$. In particular, the condition $FG=0$ implies that $\mathbf{e_1}^TG=0$. The matrix $N=\left[
\begin{array}{ccc}0&0&0\\0&0&0\\0&0&\mathbf{e_m}\mathbf{e_1}^T
\end{array}
\right]$ is then nilpotent and it commutes with $A$ and with $B$. Moreover, for each $\lambda \ne 0$ there exists a linear combination of $A$ and $C+\lambda N$ which has rank at least $l+1$, so $(A,B,C+\lambda N)\in \overline{R(3,n)}$ by the assumption of the proposition, and therefore $(A,B,C)\in \overline{R(3,n)}$.

We can now assume that for all $\mu ,\nu \in \mathbb{F}$ we have $\mu F+\nu F'=0$ or $\mu G+\nu G'=0$, which is equivalent to $F=F'=0$ or $G=G'=0$, and by Lemma \ref{reduction}(e) we can assume that $G=G'=0$. The triple $(A,B,C)$ belongs to the $2l(m+l)$-dimensional affine space of all triples of $n\times n$ matrices such that the first matrix of the triple is equal to $A$ and the other two matrices have nonzero entries only in the last two block columns of the first block row. This is an irreducible subvariety of $N(3,n)$, therefore by Lemma \ref{irr-open} we can assume any open condition on $A$, $B$ and $C$, and we assume that some linear combination of $F$ and $F'$ is of full rank (i.e. equal to $\min \{m,l\}$).

If $m\ge 2$, then we can use the previous lemma to obtain a nonzero nilpotent $n\times n$ matrix of the form $N=\left[
\begin{array}{ccc}N_1&0&0\\0&N_1&0\\0&N_2&N_3
\end{array}
\right]$ that commutes with $A$ and $B$, therefore $(A,B,C+\lambda N)\in N(3,n)$ for all $\lambda \in \mathbb{F}$. Since $N$ is nonzero and some linear combination of $F$ and $F'$ has full rank, the triple $(A,B,C+\lambda N)$ belongs to $\overline{R(3,n)}$ for all $\lambda \ne 0$ by the cases already considered in this proof. Therefore the triple $(A,B,C)$ also belongs to $\overline{R(3,n)}$.

However, if $m=1$, then by Corollary 2.1 of \cite{W} there exists some nonzero linear combination of $A$, $B$ and $C$ that has rank at most $l-1$, and by Lemma \ref{reduction}(b) we can assume that $\mathrm{rank}\, B\le l-1$. Therefore there exists some nonzero vector $\mathbf{x}\in \mathbb{F}^l$ such that $\mathbf{x}^TE=0$ and $\mathbf{x}^TF=0$. Moreover, since $[E\quad F]$ is $l\times (l+1)$ matrix of rank at most $l-1$, its kernel is at least 2-dimensional. In particular, there exist $\mathbf{y}\in \mathbb{F}^l$ and $\zeta \in \mathbb{F}$, not both of them zero, such that $\mathbf{x}^T\mathbf{y}=0$ and $E\mathbf{y}+\zeta F=0$. The matrix $N=\left[
\begin{array}{ccc}\mathbf{y}\mathbf{x}^T&0&0\\0&\mathbf{y} \mathbf{x}^T&0\\0&\zeta \mathbf{x}^T&0
\end{array}
\right]$ is then nilpotent and it commutes with $A$ and $B$. As in the previous paragraph the triple $(A,B,C+\lambda N)$ belongs to $\overline{R(3,n)}$ for each $\lambda \ne 0$, therefore $(A,B,C)\in \overline{R(3,n)}$, which completes the proof of the proposition.
\end{proof}

To prove irreducibility of the variety $N(3,6)$ only the case of Jordan blocks of orders 3, 2 and 1 is left to be considered. We will prove that triples of commuting nilpotent matrices with such Jordan structure also belong to $\overline{R(3,6)}$. However, if the matrix $A$ is such a matrix, then the variety $N_2(A)$ is reducible, cf. Proposition \ref{N_2red}. This means that we have to perturb all three matrices in the triple, not only two of them, as it was done in all previous cases. The perturbed matrices can be computed explicitly, but expressing them in terms of the original matrices is complicated. Therefore we prefer to give a geometric proof of their existence. 

\begin{prop}\label{321}
Let $(A,B,C)$ be a triple of commuting nilpotent $6\times 6$ matrices. If the Jordan canonical form of $A$ has one Jordan block of order 3, one Jordan block of order 2 and one zero Jordan block, then the triple $(A,B,C)$ belongs to $\overline{R(3,6)}$.
\end{prop}
\begin{proof}
By Lemma \ref{reduction}(a) we can assume that the matrix $A$ is in the Jordan canonical form. Since the matrices $B$ and $C$ are nilpotent and they commute with $A$, they look like $B=\left[
\begin{array}{cccccc}0&a&b&c&d&e\\0&0&a&0&c&0\\0&0&0&0&0&0\\ 0&f&g&0&h&i\\0&0&f&0&0&0\\0&0&j&0&k&0
\end{array}
\right]$ and $C=\left[
\begin{array}{cccccc}0&a'&b'&c'&d'&e'\\0&0&a'&0&c'&0\\0&0&0&0&0&0 \\ 0&f'&g'&0&h'&i'\\0&0&f'&0&0&0\\0&0&j'&0&k'&0
\end{array}
\right]$. If there is some $\lambda \in \mathbb{F}$ such that $c+\lambda c'$ and $f+\lambda f'$ are both nonzero, then the matrix $B+\lambda C$ has rank at least 4 and $(A,B,C)\in \overline{R(3,6)}$ by Corollary \ref{2-regN(3,n)irr}. Therefore we can in the sequel assume that $c=c'=0$ or $f=f'=0$, and by Lemma \ref{reduction}(e) we can assume that $f=f'=0$. We will consider two cases.

{\bf Case 1.} Assume that $c\ne 0$ or $c'\ne 0$. By Lemma \ref{reduction}(b) we can assume that $c\ne 0$. If $i+\lambda i'\ne 0$ for some $\lambda \in \mathbb{F}$, then we can by Lemma \ref{reduction}(b) assume that $i\ne 0$. However, then $\dim (\ker A\cap \ker B)=1$, so Theorem 2 of \cite{K} implies that $\dim \mathbb{F}[A,B]=6$, i.e. $(B,C)\in D_2(A)$, and therefore $(A,B,C)\in \overline{R(3,6)}$ (see the proof of Corollary \ref{2-regN(3,n)irr}). In the rest of Case 1 we can therefore assume that $i=i'=0$. To simplify the notation of the matrices we will change the basis of $\mathbb{F}^6$. Let $P=\left[
\begin{array}{cccccc}1&0&0&0&0&0\\0&1&0&0&0&0\\0&0&0&0&1&0\\ 0&0&1&0&0&0\\0&0&0&0&0&1\\0&0&0&1&0&0
\end{array}
\right]$. By Lemma \ref{reduction}(a) the triple $(A,B,C)$ belongs to $\overline{R(3,6)}$ if and only if the triple $(P^{-1}AP,P^{-1}BP,P^{-1}CP)$ does, therefore we will assume that the matrices $A$, $B$ and $C$ are in the following block form:
$$A=\left[
\begin{array}{ccc}0&\mathbf{e_1}^T&0\\0&0&\widetilde{I}\\0&0&0
\end{array}
\right] ,\quad B=\left[
\begin{array}{ccc}0&\mathbf{b}^T&\mathbf{b'}^T\\0&0& \widetilde{B}\\0&0&0
\end{array}
\right] \quad \mathrm{and}\quad C=\left[
\begin{array}{ccc}0&\mathbf{c}^T&\mathbf{c'}^T\\0&0& \widetilde{C}\\0&0&0
\end{array}
\right]$$
where the first rows and columns are of dimension 1, the second ones of dimension 3 and the last ones of dimension 2, and
$$\widetilde{I}=\left[
\begin{array}{cc}1&0\\0&1\\0&0
\end{array}
\right] ,\quad \widetilde{B}=\left[
\begin{array}{cc}a&c\\g&h\\j&k
\end{array}
\right] ,\quad \widetilde{C}=\left[
\begin{array}{cc}a'&c'\\g'&h'\\j'&k'
\end{array}
\right] ,$$
$$\mathbf{b}=\left[
\begin{array}{c}a\\c\\e
\end{array}
\right] ,\mathbf{c}=\left[
\begin{array}{c}a'\\c'\\e'
\end{array}
\right] ,\mathbf{b'}=\left[
\begin{array}{c}b\\d
\end{array}
\right] ,\mathbf{c'}=\left[
\begin{array}{c}b'\\d'
\end{array}
\right] .$$
The commutativity relation of $B$ and $C$ is equivalent to $\mathbf{b}^T\widetilde{C}=\mathbf{c}^T \widetilde{B}$. Since $\mathbf{b}^T\mathbf{e_2}\ne 0$, the entries of the second row of $\widetilde{C}$ can be expressed as rational functions in $\mathbf{b}$, $\mathbf{c}$, $\widetilde{B}$ and the other two rows of $\widetilde{C}$, therefore the triple $(A,B,C)$ belongs to a rationally parameterized subvariety of $N(3,6)$, which is irreducible by Proposition 6 in Chapter 4, \S 5 of \cite{CLO'S}. In particular, by Lemma \ref{irr-open} we can assume any open condition on the matrices $A$, $B$ and $C$. We assume that the vectors $\mathbf{e_1}$, $\mathbf{b}$ and $\mathbf{c}$ are linearly independent. Therefore there exists a matrix $Q\in M_{3\times 2}(\mathbb{F})$ such that $\mathbf{e_1}^TQ=0$, $\mathbf{b}^TQ=\mathbf{b'}^T$ and $\mathbf{c}^TQ=\mathbf{c'}^T$. Let $S=\left[
\begin{array}{ccc}1&0&0\\0&I&Q\\0&0&I
\end{array}
\right]$. Since by Lemma \ref{reduction}(a) the triple $(A,B,C)$ belongs to $\overline{R(3,6)}$ if and only if the triple $(SAS^{-1},SBS^{-1},SCS^{-1})$ does, we can assume that $\mathbf{b'}=\mathbf{c'}=0$.

Let $\mathcal{V}$ be the variety of all triples $(X_1,X_2,X_3)\in N(3,6)$ such that for each $i=1,2,3$ the matrix $X_i$ is of the form
\begin{equation}\label{X_i}
X_i=\left[
\begin{array}{ccc}0&\mathbf{x_i}^T&0\\0&Y_i&Z_i\\0&0&\zeta _iJ
\end{array}
\right]
\end{equation}
for some $\mathbf{x_i}\in \mathbb{F}^3$, $Y_i\in M_3(\mathbb{F})$, $Z_i\in M_{3\times 2}(\mathbb{F})$ and $\zeta _i\in \mathbb{F}$, where $J$ denotes the nilpotent $2\times 2$ Jordan block. Furthermore, let $\mathcal{U}$ be the open subset of $\mathcal{V}$ consisting of all triples $(X_1,X_2,X_3)\in \mathcal{V}$ with $X_1$ 1-regular and $\mathbf{x_1},\mathbf{x_2},\mathbf{x_3}$ linearly independent, let $\mathcal{U}_1$ be the set of all $6\times 6$ matrices $X_1$ such that $(X_1,X_2,X_3)\in \mathcal{U}$ for some $X_2,X_3\in M_6(\mathbb{F})$, and let $\mathcal{U}_1'$ be the set of all 1-regular nilpotent matrices of the form (\ref{X_i}). Clearly $\mathcal{U}_1\subseteq \mathcal{U}_1'$ and it is also clear that the matrix $X_i$ of the form (\ref{X_i}) is nilpotent if and only if $Y_i$ is nilpotent. Moreover, since $X_1^5=\left[
\begin{array}{ccc}0&0&\zeta _1\mathbf{x_1}^TY_1^2Z_1J\\0&0&0\\0&0&0
\end{array}
\right]$, 1-regularity of $X_1$ implies $\zeta _1\mathbf{x_1}^TY_1^2Z_1\mathbf{e_1}\ne 0$, therefore we can define rational map $\varphi \colon \mathcal{U}_1'\times \mathbb{F}^6\to \mathcal{V}$ by
$$\varphi (X,\alpha ,\beta ,\gamma ,\alpha ',\beta ',\gamma ')=\big(X,\alpha X+\beta X^2+\gamma X^3-\frac{\beta \mathbf{x}^TZ\mathbf{e_1}+\gamma \mathbf{x}^TYZ\mathbf{e_1}}{\mathbf{x}^TY^2Z\mathbf{e_1}}X^4+$$
$$(\frac{(\beta \mathbf{x}^TZ\mathbf{e_1}+\gamma \mathbf{x}^TYZ\mathbf{e_1})(\zeta \mathbf{x}^TYZ\mathbf{e_1}+\mathbf{x}^TY^2Z \mathbf{e_2})}{\zeta (\mathbf{x}^TY^2Z\mathbf{e_1})^2}-\frac{\beta \mathbf{x}^TZ\mathbf{e_2}+\gamma (\zeta \mathbf{x}^TZ\mathbf{e_1}+\mathbf{x}^TYZ \mathbf{e_2})} {\zeta \mathbf{x}^TY^2Z\mathbf{e_1}})X^5,$$
$$\alpha 'X+\beta 'X^2+\gamma 'X^3-\frac{\beta ' \mathbf{x}^TZ\mathbf{e_1}+\gamma ' \mathbf{x}^TYZ\mathbf{e_1}}{\mathbf{x}^TY^2Z\mathbf{e_1}}X^4+$$
$$(\frac{(\beta '\mathbf{x}^TZ\mathbf{e_1}+\gamma ' \mathbf{x}^TYZ\mathbf{e_1})(\zeta \mathbf{x}^TYZ\mathbf{e_1}+\mathbf{x}^TY^2Z \mathbf{e_2})}{\zeta (\mathbf{x}^TY^2Z\mathbf{e_1})^2}-\frac{\beta ' \mathbf{x}^TZ\mathbf{e_2}+\gamma '(\zeta \mathbf{x}^TZ\mathbf{e_1}+\mathbf{x}^TYZ \mathbf{e_2})} {\zeta \mathbf{x}^TY^2Z\mathbf{e_1}})X^5\big)$$
where $X=\left[
\begin{array}{ccc}0&\mathbf{x}^T&0\\0&Y&Z\\0&0&\zeta J
\end{array}
\right]$. This map is clearly injective.

If $(X_1,X_2,X_3)\in \mathcal{U}$, then $X_2$ and $X_3$ are polynomials in 1-regular matrix $X_1$, i.e. $(X_1,X_2,X_3)=\varphi (X_1,\alpha ,\beta ,\gamma ,\alpha ',\beta ',\gamma ')$ for some $\alpha ,\beta ,\gamma ,\alpha ',\beta ',\gamma '\in \mathbb{F}$. Since $\mathbf{x_1}$, $\mathbf{x_2}$ and $\mathbf{x_3}$ are linearly independent, we obtain that $\mathbf{x_1}^T$, $\mathbf{x_1}^TY_1$ and $\mathbf{x_1}^TY_1^2$ are linearly independent and that $(\beta ,\gamma )$ and $(\beta ',\gamma ')$ are linearly independent. Conversely, if a nilpotent matrix $X_1$ of the form (\ref{X_i}) is 1-regular, $\mathbf{x_1}^T,\mathbf{x_1}^TY_1,\mathbf{x_1}^TY_1^2$ are linearly independent, $(\beta ,\gamma ),(\beta ',\gamma')\in \mathbb{F}^2$ are linearly independent and $\alpha ,\alpha '\in \mathbb{F}$ are arbitrary, then $\varphi (X_1,\alpha ,\beta ,\gamma ,\alpha ',\beta ',\gamma ')\in \mathcal{U}$. Therefore $\mathcal{U}_1$ consists of all 1-regular nilpotent matrices of the form (\ref{X_i}) such that $\mathbf{x_1}^T$, $\mathbf{x_1}^TY_1$ and $\mathbf{x_1}^TY_1^2$ are linearly independent, and $\mathcal{U}=\varphi (\mathcal{U}_1\times \mathcal{U}_2)$ where $\mathcal{U}_2=\{(\alpha ,\beta ,\gamma ,\alpha ',\beta ',\gamma ')\in \mathbb{F}^6;\beta \gamma '\ne \beta '\gamma\}$. In particular, $\mathcal{U}_2$ is open in $\mathbb{F}^6$ and $\mathcal{U}_1$ is open in the variety of all nilpotent matrices of the form (\ref{X_i}) which is isomorphic to $\mathbb{F}^{10}\times N_3$. Moreover, $\mathcal{U}_1$ is nonempty, since it contains the nilpotent $6\times 6$ Jordan block. $\mathbb{F}^{10}\times N_3$ is irreducible, therefore $\overline{\mathcal{U}_1}$ is irreducible and $\dim \overline{\mathcal{U}_1}=10+\dim N_3=16$.

Since the variety $\overline{\mathcal{U}_1}\times \mathbb{F}^6$ is irreducible, Proposition 14 of \cite{S3} implies that the variety $\overline{\mathcal{U}}=\overline{\varphi (\mathcal{U}_1\times \mathcal{U}_2)}$ is also irreducible. Moreover, since $\varphi$ is injective, Theorem 11.12 of \cite{Har} implies that $\dim \overline{\mathcal{U}}=6+\dim \overline{\mathcal{U}_1}=22$.

Now we define
$$\mathcal{V'}=\{(\mathbf{x_1},Z_1,\mathbf{x_2},Z_2,\mathbf{x_3},Z_3)\in (\mathbb{F}^3\times M_{3\times 2}(\mathbb{F}))^3;\mathbf{x_i}^TZ_j=\mathbf{x_j}^TZ_i\, \mathrm{for}\, \mathrm{all}\, \mathrm{distinct}\, i\, \mathrm{and}\, j\}.$$
Note that this variety is isomorphic to the subvariety of $\mathcal{V}$ consisting of all triples of the form
$$\left(\left[
\begin{array}{ccc}0&\mathbf{x_1}^T&0\\0&0&Z_1\\0&0&0
\end{array}
\right] ,\left[
\begin{array}{ccc}0&\mathbf{x_2}^T&0\\0&0&Z_2\\0&0&0
\end{array}
\right] ,\left[
\begin{array}{ccc}0&\mathbf{x_3}^T&0\\0&0&Z_3\\0&0&0
\end{array}
\right]\right)$$
for some $\mathbf{x_i}\in \mathbb{F}^3$ and some $Z_i\in M_{3\times 2}(\mathbb{F})$. Furthermore, let $\mathcal{U'}$ be the subset of $\mathcal{V'}$ consisting of such 6-tuples in $\mathcal{V'}$ that $\mathbf{x_1}$, $\mathbf{x_2}$ and $\mathbf{x_3}$ are linearly independent, and let
$$\mathcal{W}=\{(Z_1,Z_2,Z_3)\in M_{3\times 2}(\mathbb{F})^3;\mathbf{e_i}^TZ_j=\mathbf{e_j}^TZ_i\, \mathrm{for}\, \mathrm{all}\, \mathrm{distinct}\, i\, \mathrm{and}\, j\}.$$
The variety $\mathcal{W}$ is a vector space, therefore it is irreducible and of dimension 12. Moreover, if $(\mathbf{x_1},Z_1,\mathbf{x_2},Z_2,\mathbf{x_3},Z_3)\in \mathcal{U'}$, then there exists a unique matrix $S\in GL_3(\mathbb{F})$ such that $\mathbf{e_i}^TS=\mathbf{x_i}^T$ for $i=1,2,3$. The rational map $\psi \colon GL_3(\mathbb{F})\times \mathcal{W}\to \overline{\mathcal{U'}}$ defined by $\psi (S,Z_1,Z_2,Z_3)=(S^T\mathbf{e_1},S^{-1}Z_1,S^T \mathbf{e_2},S^{-1}Z_2,S^T\mathbf{e_3},S^{-1}Z_3)$ is therefore a birational equivalence, so $\overline{\mathcal{U'}}$ is irreducible by Proposition 6 in Chapter 4, \S 5 of \cite{CLO'S}, and $\dim \overline{\mathcal{U'}}=\dim \mathcal{W}+\dim GL_3(\mathbb{F})=21$ by Corollary 7 in Chapter 9, \S 5 of \cite{CLO'S}.

Now we define the projection $\pi \colon \overline{\mathcal{U}}\to \overline{\mathcal{U'}}$ by
$$\pi \left(\left[
\begin{array}{ccc}0&\mathbf{x_1}^T&0\\0&Y_1&Z_1\\0&0&\zeta _1J
\end{array}
\right] ,\left[
\begin{array}{ccc}0&\mathbf{x_2}^T&0\\0&Y_2&Z_2\\0&0&\zeta _2J
\end{array}
\right] ,\left[
\begin{array}{ccc}0&\mathbf{x_3}^T&0\\0&Y_3&Z_3\\0&0&\zeta _3J
\end{array}
\right]\right) =(\mathbf{x_1},Z_1,\mathbf{x_2},Z_2,\mathbf{x_3},Z_3).$$
Assume that the projection $\pi$ is not dominant (i.e. the image $\pi (\overline{\mathcal{U}})$ is not dense in $\overline{\mathcal{U'}}$). Then $\overline{\pi (\overline{\mathcal{U}})}$ is a proper subvariety of $\overline{\mathcal{U'}}$, and the irreducibility of $\overline{\mathcal{U'}}$ implies that $\dim \overline{\pi (\overline{\mathcal{U}})}\le \dim \overline{\mathcal{U'}}-1=20$. Theorem 11.12 of \cite{Har} then implies that any component of any fibre of the projection $\pi$ has dimension at least $\dim \overline{\mathcal{U}}-\dim \overline{\pi (\overline{\mathcal{U}})}\ge 2$.

However, let $\mathbf{x_i}=\mathbf{e_i}$ for $i=1,2,3$,
$$Z_1=\left[
\begin{array}{cc}0&1\\0&0\\1&0
\end{array}
\right] ,\quad Z_2=\left[
\begin{array}{cc}0&0\\1&0\\0&1
\end{array}
\right] \quad \mathrm{and}\quad Z_3=\left[
\begin{array}{cc}1&0\\0&1\\0&0
\end{array}
\right] ,$$
and assume that $\pi (X_1,X_2,X_3)=(\mathbf{x_1},Z_1,\mathbf{x_2},Z_2,\mathbf{x_3},Z_3)$ where for each $i=1,2,3$, $X_i=\left[
\begin{array}{ccc}0&\mathbf{x_i}^T&0\\0&Y_i&Z_i\\0&0&\zeta _iJ
\end{array}
\right]$ for some $Y_i\in M_3(\mathbb{F})$ and $\zeta _i\in \mathbb{F}$, and at least one of $Y_1,Y_2,Y_3,\zeta _1,\zeta _2,\zeta _3$ is nonzero. Then
\begin{eqnarray}
\mathbf{x_i}^TY_j&=&\mathbf{x_j}^TY_i,\label{x_iY_j}\\
Y_iZ_j+\zeta _jZ_iJ&=&Y_jZ_i+\zeta _iZ_jJ,\label{Y_iZ_j}\\
Y_iY_j&=&Y_jY_i\label{Y_iY_j}
\end{eqnarray}
for all distinct $i$ and $j$. Moreover, the matrices $Y_i$ are nilpotent and in particular
\begin{equation}\label{Tr}
\mathrm{Tr}\, (Y_i)=0
\end{equation}
for each $i=1,2,3$. Solving linear equations (\ref{x_iY_j}), (\ref{Y_iZ_j}) and (\ref{Tr}) we obtain
$$Y_1=\left[
\begin{array}{ccc}\alpha&\beta&\gamma\\\gamma&-2\alpha&\beta\\ \beta -\delta&\gamma&\alpha
\end{array}
\right] ,Y_2=\left[
\begin{array}{ccc}\gamma&-2\alpha&\beta\\\beta -\delta&-2\gamma&-2\alpha\\-2\alpha&\beta -\delta&\gamma
\end{array}
\right] ,Y_3=\left[
\begin{array}{ccc}\beta -\delta&\gamma&\alpha\\-2\alpha&\beta -\delta&\gamma\\4\gamma&-2\alpha&\beta -\delta
\end{array}
\right] ,$$
$\zeta _1=\delta ,\zeta _2=-3\alpha ,\zeta _3=3\gamma$ and $3(\beta -\delta )=0$. Now we have to consider different cases for the characteristic of the field $\mathbb{F}$.

If $\mathrm{char}\, \mathbb{F}\ne 2,3$, then first we get $\delta =\beta$. A short calculation shows that (\ref{Y_iY_j}) is equivalent to the equations $2\alpha \beta =3\gamma^2$ and $2\beta \gamma =-3\alpha ^2$. Since at least one of $\alpha$, $\beta$ and $\gamma$ is nonzero, it follows that $\beta \ne 0$ and
$$Y_1=\left[
\begin{array}{ccc}\beta \omega&\beta&-\frac{3}{2}\beta \omega ^2\\-\frac{3}{2}\beta \omega ^2&-2\beta \omega&\beta\\0&-\frac{3}{2}\beta \omega ^2&\beta \omega
\end{array}
\right] ,\quad Y_2=\left[
\begin{array}{ccc}-\frac{3}{2}\beta \omega ^2&-2\beta \omega&\beta\\0&3\beta \omega ^2&-2\beta \omega\\-2\beta \omega&0&-\frac{3}{2}\beta \omega ^2
\end{array}
\right] ,$$
$$Y_3=\left[
\begin{array}{ccc}0&-\frac{3}{2}\beta \omega ^2&\beta \omega\\-2\beta \omega&0&-\frac{3}{2}\beta \omega ^2\\-6\beta \omega ^2&-2\beta \omega&0
\end{array}
\right] ,$$
$\zeta _1=\beta$, $\zeta _2=-3\beta \omega$ and $\zeta _3=-\frac{9}{2}\beta \omega ^2$, where $\omega (27\omega ^3-8)=0$. It can be verified that the matrices $Y_1$, $Y_2$ and $Y_3$ are indeed nilpotent and $X_1$ is 1-regular for $\beta \ne 0$. The fibre $\pi ^{-1}(\mathbf{x_1},Z_1,\mathbf{x_2},Z_2,\mathbf{x_3},Z_3)$ is therefore nonempty and it is equal to a union of a finite number of 1-dimensional affine spaces, hence it is 1-dimensional, which is a contradiction.

If $\mathrm{char}\, \mathbb{F}=2$, then we again obtain $\delta =\beta$. The matrix $Y_2$ is then upper triangular with the diagonal $(\gamma ,0,\gamma )$, and since it is nilpotent, it follows that $\gamma =0$. However, then the matrix $Y_1$ is upper triangular with the diagonal $(\alpha ,0,\alpha )$, therefore $\alpha =0$. Hence $Y_1=\left[
\begin{array}{ccc}0&\beta&0\\0&0&\beta\\0&0&0
\end{array}
\right]$, $Y_2=\left[
\begin{array}{ccc}0&0&\beta\\0&0&0\\0&0&0
\end{array}
\right]$, $Y_3=0$, $\zeta _1=\beta$ and $\zeta _2=\zeta _3=0$, i.e. the fibre $\pi ^{-1}(\mathbf{x_1},Z_1,\mathbf{x_2},Z_2,\mathbf{x_3},Z_3)$ is 1-dimensional affine space, which is again a contradiction.

However, if $\mathrm{char}\, \mathbb{F}=3$, then the relation (\ref{Y_iY_j}) implies that $\alpha \delta =\gamma \delta =0$. Assume that $(X_1,X_2,X_3)\in \mathcal{U}$. Then $\delta \ne 0$, so $\alpha =\gamma =0$. However, then $Y_3=(\beta -\delta )I$, so $\delta =\beta$ and we get the same solution as in the case of $\mathrm{char}\, \mathbb{F}=2$. Therefore $\dim (\mathcal{U}\cap \pi ^{-1}(\mathbf{x_1},Z_1,\mathbf{x_2},Z_2,\mathbf{x_3},Z_3))=1$. However, this set is open in the fibre $\pi ^{-1}(\mathbf{x_1},Z_1,\mathbf{x_2},Z_2,\mathbf{x_3},Z_3)$, therefore its closure is a union of some irreducible components of $\pi ^{-1}(\mathbf{x_1},Z_1,\mathbf{x_2},Z_2,\mathbf{x_3},Z_3)$. Hence, there exists a component of this fibre which is 1-dimensional, which is again a contradiction.

In each characteristic we obtained a contradiction, therefore the projection $\pi$ is dominant. Then the continuity of the projection $\pi$ implies that $\overline{\mathcal{U'}}=\overline{\pi (\overline{\mathcal{U}})}\subseteq \overline{\pi (\mathcal{U})}$, i.e. $\pi (\mathcal{U})$ is also dense in $\overline{\mathcal{U'}}$. If $(\mathbf{x_1},Z_1,\mathbf{x_2},Z_2,\mathbf{x_3},Z_3)\in \pi (\mathcal{U})$ is arbitrary, then there exist $Y_1,Y_2,Y_3\in N_3$ and $\zeta _1,\zeta _2,\zeta _3\in \mathbb{F}$ such that
$$\left(\left[
\begin{array}{ccc}0&\mathbf{x_1}^T&0\\0&Y_1&Z_1\\0&0&\zeta _1J
\end{array}
\right], \left[
\begin{array}{ccc}0&\mathbf{x_2}^T&0\\0&Y_2&Z_2\\0&0&\zeta _2J
\end{array}
\right] ,\left[
\begin{array}{ccc}0&\mathbf{x_3}^T&0\\0&Y_3&Z_3\\0&0&\zeta _3J
\end{array}
\right]\right)\in \mathcal{U}\subseteq R_1(3,6).$$
Moreover, since the line $\left\{ \left[
\begin{array}{ccc}0&\mathbf{x_1}^T&0\\0&\lambda Y_1&Z_1\\0&0&\lambda \zeta _1J
\end{array}
\right] ;\lambda \in \mathbb{F}\right\}$ intersects the open set of all 1-regular matrices, all except finitely many matrices on this line are 1-regular, i.e.
$$\left(\left[
\begin{array}{ccc}0&\mathbf{x_1}^T&0\\0&\lambda Y_1&Z_1\\0&0&\lambda \zeta _1J
\end{array}
\right], \left[
\begin{array}{ccc}0&\mathbf{x_2}^T&0\\0&\lambda Y_2&Z_2\\0&0&\lambda \zeta _2J
\end{array}
\right] ,\left[
\begin{array}{ccc}0&\mathbf{x_3}^T&0\\0&\lambda Y_3&Z_3\\0&0&\lambda \zeta _3J
\end{array}
\right]\right) \in R_1(3,6)$$
for all except finitely many scalars $\lambda \in \mathbb{F}$, and therefore
$$\left(\left[
\begin{array}{ccc}0&\mathbf{x_1}^T&0\\0&0&Z_1\\0&0&0
\end{array}
\right], \left[
\begin{array}{ccc}0&\mathbf{x_2}^T&0\\0&0&Z_2\\0&0&0
\end{array}
\right] ,\left[
\begin{array}{ccc}0&\mathbf{x_3}^T&0\\0&0&Z_3\\0&0&0
\end{array}
\right]\right)\in \overline{R(3,6)}.$$
We proved that the set 
$$\left\{\left(\left[
\begin{array}{ccc}0&\mathbf{x_1}^T&0\\0&0&Z_1\\0&0&0
\end{array}
\right], \left[
\begin{array}{ccc}0&\mathbf{x_2}^T&0\\0&0&Z_2\\0&0&0
\end{array}
\right] ,\left[
\begin{array}{ccc}0&\mathbf{x_3}^T&0\\0&0&Z_3\\0&0&0
\end{array}
\right]\right) ;(\mathbf{x_1},Z_1,\mathbf{x_2},Z_2,\mathbf{x_3},Z_3)\in \pi (\mathcal{U})\right\}$$
is a subset of $\overline{R(3,6)}$, and since $\overline{\pi (\mathcal{U})}=\overline{\mathcal{U'}}$, the set
$$\left\{\left(\left[
\begin{array}{ccc}0&\mathbf{x_1}^T&0\\0&0&Z_1\\0&0&0
\end{array}
\right], \left[
\begin{array}{ccc}0&\mathbf{x_2}^T&0\\0&0&Z_2\\0&0&0
\end{array}
\right] ,\left[
\begin{array}{ccc}0&\mathbf{x_3}^T&0\\0&0&Z_3\\0&0&0
\end{array}
\right]\right) ;(\mathbf{x_1},Z_1,\mathbf{x_2},Z_2,\mathbf{x_3},Z_3)\in \overline{\mathcal{U'}}\right\}$$
is also a subset of $\overline{R(3,6)}$. In particular $(A,B,C)\in \overline{R(3,6)}$, which proves Case 1.

{\bf Case 2.} Assume that $c=c'=0$. If there exists $\lambda \in \mathbb{F}$ such that $i+\lambda i'$ and $k+\lambda k'$ are both nonzero, then there exists $\mu \in \mathbb{F}$ such that $\mathrm{rank}\, (\mu A+B+\lambda C)=4$, and $(A,B,C)\in \overline{R(3,6)}$ by Corollary \ref{2-regN(3,n)irr}. Therefore we can assume that $i=i'=0$ or $k=k'=0$, and by Lemma \ref{reduction}(e) we can assume that $k=k'=0$. The commutativity relation of $B$ and $C$ is then equivalent to the equations $ej'=e'j$ and $ij'=i'j$. If $j\ne 0$ or $j'\ne 0$, then by Lemma \ref{reduction}(b) we can assume that $j=1$ and $j'=0$, and the commutativity relation of $B$ and $C$ implies that $e'=i'=0$. The triple $(A,B,C)$ then belongs to some affine space, therefore by Lemma \ref{irr-open} we can assume that $i\ne 0$. The matrix $X=\mathbf{e_6}\mathbf{e_5}^T$ commutes with $A$ and with $C$, therefore $(A,B+\lambda X,C)\in N(3,6)$ for each $\lambda \in \mathbb{F}$. Moreover, for each $\lambda \ne 0$ there exists $\mu \in \mathbb{F}$ such that $\mathrm{rank}\, (B+\lambda X+\mu A)=4$, therefore the triple $(A,B+\lambda X,C)$ belongs to $\overline{R(3,6)}$ by Corollary \ref{2-regN(3,n)irr}. Hence $(A,B,C)\in \overline{R(3,6)}$.

It remains to consider the case when $j=j'=0$. By Lemma \ref{reduction}(b) we can exchange $B$ and $C$ or add any multiple of $B$ to $C$, therefore we can assume that $i'=0$. The triple $(A,B,C)$ then belongs to a 13-dimensional affine space which is an irreducible subvariety of $N(3,6)$, and by Lemma \ref{irr-open} we can assume any open condition on $B$ and $C$. We will assume that $e'\ne 0$. Then the matrix $X=\left[
\begin{array}{cccccc}0&0&0&e'&0&0\\0&0&0&0&e'&0\\0&0&0&0&0&0\\ 0&0&0&0&0&0\\0&0&0&0&0&0\\0&0&g'&0&h'&0
\end{array}
\right]$ commutes with $A$ and with $C$, and $(A,B+\lambda X,C)\in \overline{R(3,6)}$ for each $\lambda \ne 0$ by Case 1. This implies that the triple $(A,B,C)$ also belongs to $\overline{R(3,6)}$, which completes the proof of the proposition.
\end{proof}

Now we can prove the main theorem of this section: irreducibility of $N(3,n)$ for $n\le 6$. Note that the new cases are $n=5$ and $n=6$, since irreducibility of $N(3,n)$ for $n\le 3$ was proved in \cite{N} and irreducibility of $N(3,4)$ was proved in \cite{Y}.

\begin{thm}\label{35,36}
The varieties $N(3,n)$ are irreducible for $n\le 6$.
\end{thm}
\begin{proof}
We have to prove that each triple from $N(3,n)$ belongs to $\overline{R(3,n)}$. Let $(A,B,C)\in N(3,n)$ be any triple. By Lemma \ref{reduction}(b) we can assume that  for each $i\le n-1$ the rank of $A^i$ is not smaller than the rank of the $i$-th power of any linear combination of $A$, $B$ and $C$. If $A$ is 2-regular, then $(A,B,C)\in \overline{R(3,n)}$ by Corollary \ref{2-regN(3,n)irr}. Otherwise the Jordan canonical form of $A$ either has only one nonzero Jordan block or all Jordan blocks of order at most 2 or Jordan blocks of sizes 3, 2 and 1. In all cases, the tripe $(A,B,C)$ belongs to $\overline{R(3,n)}$ by respectively Propositions \ref{1nonzero}, \ref{square-zero} and \ref{321}.
\end{proof}


\begin{thebibliography}{100}
\bibitem{Bar}V. Baranovsky, {\em The variety of pairs of commuting nilpotent matrices is irreducible}, Transformation groups, 6 (2001), 3--8
\bibitem{BH}J. Barria, P. R. Halmos, {\em Vector bases for two commuting matrices}, Linear and multilinear algebra, 27 (1990), 147--157
\bibitem{Bas}R. Basili, {\em On the irreducibility of commuting varieties of nilpotent matrices}, Journal of algebra, 268 (2003), 58--80
\bibitem{CLO'S}D. Cox, J. Little, D. O'Shea, {\em Ideals, varieties, and algorithms}, 2nd edition, Springer-Verlag, New York, 1997
\bibitem{Ge}M. Gerstenhaber, {\em On dominance and varieties of commuting matrices}, Annals of mathematics, 73 (1961), 324--348
\bibitem{Gu}R. M. Guralnick, {\em A note on commuting pairs of matrices}, Linear and multilinear algebra, 31 (1992), 71--75
\bibitem{GN}R. M. Guralnick, N. V. Ngo, {\em Reducibility of nilpotent commuting varieties}, preprint (2013), available at http://arxiv.org/pdf/1308.2420.pdf
\bibitem{GS}R. M. Guralnick, B. A. Sethuraman, {\em Commuting pairs and triples of matrices and related varieties}, Linear algebra and its applications, 310 (2000), 139--148
\bibitem{Han}Y. Han, {\em Commuting triples of matrices}, Electronic journal of linear algebra, 13 (2005), 274--343
\bibitem{Har}J. Harris, {\em Algebraic geometry}, Graduate texts in mathematics, Springer-Verlag, New York, 1992
\bibitem{HO}J. Holbrook, M. Omladi\v{c}, {\em Approximating commuting operators}, Linear algebra and its applications, 327 (2001), 131--149
\bibitem{KN}A. A. Kirillov, Y. A. Neretin, {\em The variety $A_n$ of $n$-dimensional Lie algebra structures}, Fourteen papers translated from the Russian, American mathematical society translations,
series 2, volume 137, 1987, 21--30
\bibitem{K}T. Ko\v{s}ir, {\em On the structure of commutative matrices II}, Linear algebra and its applications, 261 (1997), 293--305
\bibitem{LL}T. J. Laffey, S. Lazarus, {\em Two-generated commutative matrix subalgebras}, Linear algebra and its applications, 147 (1991), 249--273
\bibitem{MT}T. S. Motzkin, O. Taussky, {\em Pairs of matrices with property L II}, Transactions of the American mathematical society, 80 (1955), 387--401
\bibitem{NSa}M. G. Neubauer, D. J. Saltman, {\em Two-generated commutative subalgebras of $M_n(F)$}, Journal of algebra, 164 (1994), 545--562
\bibitem{NS}M. G. Neubauer, B. A. Sethuraman, {\em Commuting pairs in the centralizers of 2-regular matrices}, Journal of algebra, 214 (1999), 174--181
\bibitem{N}N. V. Ngo, {\em Commuting varieties of $r$-tuples over Lie algebra}, Journal of pure and applied algebra, 218 (2014), 1400--1417
\bibitem{CO'MV}K. C. O'Meara, J. Clark, C. I. Vinsonhaler, {\em Advanced topics in linear algebra: weaving matrix problems through the Weyr form}, Oxford university press, New York, 2011
\bibitem{O}M. Omladi\v{c}, {\em A variety of commuting triples}, Linear algebra and its applications, 383 (2004), 233--245
\bibitem{P}A. Premet, {\em Nilpotent commuting varieties of reductive Lie algebras}, Inventiones mathematicae, 154 (2003), 653--683
\bibitem{Sh}I. R. Shafarevich, {\em Basic algebraic geometry 1}, 2nd edition, Springer-Verlag, Berlin, 1994
\bibitem{ST}D. A. Suprunenko, R. I. Tyshkevich, {\em Commutative matrices}, Academic paperbacks, Academic press, New York, 1968
\bibitem{S1}K. \v{S}ivic, {\em On varieties of commuting triples}, Linear algebra and its applications, 428 (2008), 2006--2029
\bibitem{S2}K. \v{S}ivic, {\em On varieties of commuting triples II}, Linear algebra and its applications, 437 (2012), 461--489
\bibitem{S3}K. \v{S}ivic, {\em On varieties of commuting triples III}, Linear algebra and its applications, 437 (2012), 393--460
\bibitem{W}R. Westwick, {\em Spaces of linear transformations of equal rank}, Linear algebra and its applications, 5 (1972), 49--64
\bibitem{Y}H.-W. Young, {\em Components of algebraic sets of commuting and nearly commuting matrices}, PhD thesis, University of Michigan, 2010
\end{thebibliography}
\end{document}